\documentclass{article}

\usepackage{amscd,amsmath,amssymb,amsthm}
\usepackage{color}

\newcommand{\eps}{\varepsilon}

\newcommand{\R}{\mathbb R}
\newcommand{\N}{\mathbb N}

\newcommand{\J}{{\cal J}}

\DeclareMathOperator*{\esssup}{ess\; sup}
\DeclareMathOperator*{\essinf}{ess\; inf}

\newcommand\meas{{\rm meas}}
\newcommand\e{{\rm e}}

\newtheorem{corollary}{Corollary}[section]
\newtheorem{theorem}[corollary]{Theorem}
\newtheorem{lemma}[corollary]{Lemma}
\newtheorem{proposition}[corollary]{Proposition}

\theoremstyle{definition}
\newtheorem{definition}[corollary]{Definition}
\newtheorem{remark}[corollary]{Remark}

\numberwithin{equation}{section}
\begin{document}

\title{{{\bf Radial bounded solutions\\
 for modified Schr\"odinger equations}} 
\footnote{The research that led to the present paper was partially supported 
by Horizon Europe Seeds STEPS: STEerability and controllability of PDEs in Agricolture and Physical Models, CUP: H91I21001640006.  \\
The authors are members of Research Group INdAM-GNAMPA.}}

\author{Federica Mennuni and Addolorata Salvatore\\
{\small Dipartimento di Matematica}\\
{\small Universit\`a degli Studi di Bari Aldo Moro} \\
{\small Via E. Orabona 4, 70125 Bari, Italy}\\
{\small \it federica.mennuni@uniba.it}\\
{\small \it addolorata.salvatore@uniba.it}}
\date{}

\maketitle

\begin{abstract}
We study the quasilinear equation 
\[(P)\qquad
- {\rm div} (a(x,u,\nabla u)) +A_t(x,u,\nabla u) + |u|^{p-2}u\  =\  g(x,u) \qquad \hbox{in $\R^N$,}
\]
with $N\ge 3$ and $p > 1$.  Here, we suppose $A : \R^N \times \R \times \R^N \to \R$ is a given ${C}^{1}$-Carath\'eodory function which grows as $|\xi|^p$ with $A_t(x,t,\xi) = \frac{\partial A}{\partial t}(x,t,\xi)$, $a(x,t,\xi) = \nabla_\xi A(x,t,\xi)$ and $g(x,t)$ is a given  Carath\'eodory function on $\R^N \times \R$ which grows as $|\xi|^q$ with $1<q<p$.

Suitable assumptions on $A(x,t,\xi)$ and $g(x,t)$ 
set off the variational structure of $(P)$ and its related functional $\J$
is $C^1$ on the Banach space $X = W^{1,p}(\R^N) \cap L^\infty(\R^N)$.
In order to overcome the lack of compactness, we assume
that the problem has radial symmetry, then we look for 
critical points of $\J$ restricted to $X_r$, subspace of the radial functions in $X$.

Following an approach which exploits the interaction between
the intersection norm in $X$ and the norm on $W^{1,p}(\R^N)$, we prove the existence 
of at least two weak bounded radial solutions of $(P)$, one positive and one negative, by applying a generalized version
of the Minimum Principle.
\end{abstract}



\section{Introduction} \label{secintroduction}

In this paper we look for weak radial bounded solutions for the quasilinear equation 
\begin{equation}\label{euler}
- {\rm div} (a(x,u,\nabla u)) +A_t(x,u,\nabla u) + |u|^{p-2}u\   =\  g(x,u) \qquad \hbox{in $\R^N$,}
\end{equation}
where $p > 1$ and $N\geq 3$, $A :{\R}^{N} \times \R \times {\R}^{N} \to \R$ is a ${{C}}^{1}$-Carath\'eodory function with partial derivatives
\begin{equation*} \label{derA}
A_t(x,t,\xi) =\ \frac{\partial A}{\partial t}(x,t,\xi), \quad a(x,t,\xi) =\left(\frac{\partial A}{\partial {\xi_1}}(x,t,\xi), ...  ,  \frac{\partial A}{\partial {\xi_N}}(x,t,\xi)\right)
\end{equation*}
and $g:{\R}^{N} \times \R \to \R$ is a  suitable Carath\'eodory function.

Equation \eqref{euler} generalizes quasilinear equations describing several physical phenomena such as the self-channeling of a high-power ultra short laser, or also some problems which arise in plasma physics, fluid mechanics, mechanics and in the condensed matter theory (see \cite{PSW} and references therein or also \cite{CASPO} for some model problems).

If $A(x,t, \xi)=\bar{A}|\xi|^p$  with $\bar{A}$ real constant, equation \eqref{euler} turns out to be the $p$--Laplacian equation
\begin{equation} \label{equat1.3}
-\Delta_{p}u+|u|^{p-2}u= g(x,u) \qquad \hbox{in $\R^N$.}
\end{equation}
In the case $p =2$, \eqref{equat1.3} reduces to the following Schr\"odinger equation 
\[
-\Delta u +u= g(x,u) \qquad \hbox{in $\R^N$}
\]
which is a central topic in Nonlinear Analysis (see e.g., \cite{BCS2015}, \cite{BW}, \cite{CDeVS}, \cite{CPS}, \cite{DS}, \cite{Ra1}, \cite{AS}).

Many authors studied also equation \eqref{equat1.3} in the general case $p>1$ (see, e.g., \cite{BGR}, \cite{BCS2016}, \cite{LW1}, \cite{LiuZheng}).

We note that equation \eqref{equat1.3} has a variational structure, but there is a lack of compactness as the problem is settled in the whole Euclidean space $\R^N$ and classical variational tools do not work; thus suitable assumptions on the involved functions are required.

On the contrary, even if the function $A(x,t,\xi)$ has the particular form $\frac{1}{p}A_{1}(x,t)|\xi|^p$ but the coefficient $A_{1}(x,t)$ is not constant, besides the lack of the compactness the study of equation \eqref{euler} presents another difficulty: the loss of a direct variational formulation in the space $W^{1,p}(\R^N)$. Let us point out that this problem arises also if we look for solutions verifying homogeneous Dirichlet conditions in a bounded domain $\Omega$. Indeed, the natural action functional
\[
J_1(u)\ =\  \frac1p\ \int_{\Omega} A_1(x,u)|\nabla u|^p dx + \frac1p\ \int_{\Omega} |u|^p dx - \int_{\Omega} G(x,u) dx,
\]
is not well defined in $W_0^{1,p}(\Omega)$ if $A_1(x,t)$ is unbounded with respect to $t$. Moreover, even if $A_1(x,t)$ is strictly positive and bounded with respect to $t$ but $\frac{\partial A_1}{\partial t}(x,t) \ne 0$, then $J_1$ is defined in $W_0^{1,p}(\Omega)$ but it is  G\^ateaux differentiable only along directions of 
$W_0^{1,p}(\Omega) \cap L^\infty(\Omega)$.

Thus, many authors have studied equation \eqref{euler} by using non-smooth techniques or introducing a suitable change of variable if the term $A(x,t,\xi)$ has a very particular form or giving a ``good'' definition of critical point either on bounded domains or in unbounded ones (see, e.g., \cite{AB1}, \cite{AG}, \cite{BMP}, \cite{BP},  \cite{Ca},  \cite{CD},  \cite{CJ},  \cite{CDM},  \cite{LW},  \cite{LWW},  \cite{PSW}).

More recently, Candela and Palmieri in \cite{CP1}-\cite{CP3} consider the functional
\[
\J(u)\ =\   \int_{\Omega} A(x,u,\nabla u) dx + \frac1p\ \int_{\Omega} |u|^p dx - \int_{\Omega} G(x,u) dx,
\]
defined on the Banach space $W_0^{1,p}(\Omega) \cap L^\infty(\Omega)$ equipped with the intersection norm.

Introducing a new weak Cerami--Palais--Smale condition (see Definition \ref{defwCPS}) they state some abstract critical points Theorems. Using this variational approach, the existence of at least one bounded solution of \eqref{euler} in the case $A(x,t,\xi)=\frac{1}{p}A_1(x,t)|\xi|^p$ has been stated when $g(x,t)$ grows as $|t|^q$ with $q>p$ but subcritical and the involved functions are radially symmetric in \cite{CS_radiale} or the term $|u|^{p-2}u$ is mutiplied by a weight $V(x)$ verifying suitable assumptions in \cite{CSS1} (see also \cite{MeMu_stampa} and \cite{SalSp} where a generalized $(p,q)$-Laplacian equation in $\R^N$ is studied).

Always in the presence of a suitable weight $V(x)$, the existence of solutions of equation like to \eqref{euler} has been investigated in \cite{MS_stampa} (see also \cite{MS}) if $A(x,t,\xi)$ is a more general function which grows as $|\xi|^p$ and $g(x,t)$ has a sub-$p$-linear growth of the type
\[
|g(x,t)| \leq \eta(x) |t|^{q-1}
\]
with $\eta$ a suitable measurable function and $1<q<p$.

We notice that the results stated in \cite{MS}--\cite{MS_stampa} do not cover the case $V(x)=1$, so they do not apply to the equation \eqref{euler}. Therefore, in this paper we want to look for solutions of the equation \eqref{euler} when $A(x,t,\xi)$ and $g(x,t)$ beyond some hypotheses similar to those ones required in \cite{MS_stampa} are radially symmetric in $x$. To this aim, in Lemma \ref{lemma_bmp} we will state a convergence results in $\R^N$ already proved in bounded domains by Boccardo, Murat and Puel in \cite[Lemma 5]{BMP} (see also \cite[Lemma 4.5]{MeMu_stampa}).

The paper is organized as follows. In Section \ref{secabstract} we introduce the weak Cerami--Palais--Smale condition and the related Minimum Principle (see Proposition \ref{Minimum Principle}). In Section \ref{variational} we give some preliminary assumptions on the functions $A(x,t, \xi)$ and $g(x,t)$ which allow to give a variational 
formulation of the equation \eqref{euler}. In Section \ref{secmain} we consider some further assumptions, then we state our main results (see Theorem \ref{ThExist}) and we prove some properties of the action functional $\J$ and a convergence result {\em \`a la} Boccardo-Murat-Puel in $\R^N$. Finally in Section \ref{secproof} we prove that $\J$ verifies the weak Cerami--Palais--Smale condition in the subspace $X_r$ of the radial function of $X=W^{1,p}(\R^N) \cap L^\infty(\R^N)$ and then we state the existence of two nontrivial weak radial bounded solutions, one negative and one positive, thus concluding the proof of Theorem \ref{ThExist}.


\section{Abstract tools} \label{secabstract}

In this section we denote by $(X, \|\cdot\|_X)$ a Banach space with dual space $(X',\|\cdot\|_{X'})$,  $(W,\|\cdot\|_W)$ another Banach space such that $X \hookrightarrow W$ continuously,
and by $J: X \to \R$ a given $C^1$ functional.

Nevertheless, to avoid any ambiguity, we will henceforth denote by $X$ the space equipped with its norm $\|\cdot\|_X$, while, if the norm $\|\cdot\|_W$ is involved, we will write it explicitly.

For simplicity, taking $\beta \in \R$, we say that a sequence
$(u_n)_n\subset X$ is a {\sl Cerami--Palais--Smale sequence at level $\beta$},
briefly {\sl $(CPS)_\beta$--sequence}, if
\begin{equation*}
\lim_{n \to +\infty}J(u_n) = \beta\quad\mbox{and}\quad 
\lim_{n \to +\infty}\|dJ\left(u_n\right)\|_{X'} (1 + \|u_n\|_X) = 0.
\end{equation*}

Moreover, $\beta$ is a {Cerami--Palais--Smale level}, briefly {{\sl $(CPS)$--level}, if there exists a $(CPS)_\beta$--sequence. 

The functional $J$ satisfies the classical Cerami--Palais--Smale condition in $X$ at the level $\beta$ if every sequence $(CPS)_\beta$--sequence converges in $X$ up to subsequences. However, thinking about the setting of our problem, in general a $(CPS)_\beta$--sequence may also exist which is unbounded in $\|\cdot\|_X$ but which converges with respect to $\|\cdot\|_W$. Then, we can weaken the classical Cerami--Palais--Smale condition in an appropriate way according to ideas that have already been developed in previous papers (see, for example, \cite{CP1}--\cite{CP3}).  

\begin{definition} \label{defwCPS}
The functional $J$ satisfies the
{\slshape weak Cerami--Palais--Smale 
condition at level $\beta$} ($\beta \in \R$), 
briefly {\sl $(wCPS)_\beta$ condition}, if for every $(CPS)_\beta$--sequence $(u_n)_n$,
a point $u \in X$ exists such that 
\begin{description}{}{}
\item[{\sl (i)}] $\displaystyle 
\lim_{n \to+\infty} \|u_n - u\|_W = 0\quad$ (up to subsequences),
\item[{\sl (ii)}] $J(u) = \beta$, $\; dJ(u) = 0$.
\end{description}
If $J$ satisfies the $(wCPS)_\beta$ condition at each level $\beta \in I$, $I$ real interval, 
we say that $J$ satisfies the $(wCPS)$ condition in $I$.
\end{definition}

Let us point out that, due to the convergence only in the norm of $W$, the {\sl $(wCPS)_\beta$ }condition  implies that the set of critical points of $J$ at the $\beta$ level is compact with respect to $\|\cdot\|_W$, so that we can state a Deformation Lemma and some abstract theorems about critical points (see \cite{CP3}). In particular, the following Minimum Principle applies (for the proof, see \cite[Theorem 1.6]{CP3}).

\begin{proposition}{(Minimum Principle)}
\label{Minimum Principle}
If $J \in {C}^{1}(X,\R)$ is bounded from below in $X$ and {\sl $(wCPS)_\beta$} holds at level $\beta =\displaystyle{\inf_{X} J \in \R}$, then $J$ attains its infimum, i.e.,  ${u}_0 \in X$ exists such that $J({u}_0)= \beta$.
\end{proposition}


\section{Variational setting and first properties}
\label{variational}

Here and in the following, let $\N=\{1,2,...\}$ be the set of the strictly positive integers and we denote by $x \cdot y$ the inner product in $\R^N$ and $|\cdot|$ the standard norm on any Euclidean space as the dimension of the considered vector is clear and no ambiguity arises.
Furthermore, we denote by:
\begin{itemize} 
\item ${B}_{R}(x) =  \{ y \in {\R}^{N}:\vert y-x \vert < R \}$ the open ball in ${\R}^{N}$ with center in $x \in {\R}^{N}$ and radius $R>0$;
\item ${{B}}^{c}_R=\R^N\setminus {B}_{R}(0)$ the complement of the open ball ${B}_{R}(0)$ in ${\R}^{N}$;
\item $\meas(\Omega)$ the usual Lebesgue measure of a measurable set $\Omega$ in $\R^N$;
\item $L^l(\R^N)$ the Lebesgue space with
norm $|u|_l = \displaystyle{ \left(\int_{\R^N}|u|^l dx\right)^{1/l}}$ if $1 \le l < +\infty$;
\item $L^\infty(\R^N)$ the space of Lebesgue--measurable 
and essentially bounded functions $u :\R^N \to \R$ with norm
\[
|u|_{\infty} = \esssup_{\R^N} |u|;
\]
\item $W^{1,p}(\R^N)$ the classical Sobolev space with
norm $\|u\|_{W} = \displaystyle{(|\nabla u|_p^p + |u|_p^p)^{\frac1p}}$ if $1 \le p < +\infty$;
\item $W_r^{1,p}(\R^N) = \{u \in W^{1,p}(\R^N): u(x) = u(|x|) \ \mbox{a.e.}\ x \in \R^N\}$ 
the subspace of the radial function of $W^{1,p}(\R^N)$ equipped with 
the same norm $\|\cdot\|_W$ with dual space $(W_r^{1,p}(\R^N))'$.
\end{itemize}

From the Sobolev Embedding Theorems, for any $l \in [p,p^*]$
with $p^* = \frac{pN}{N-p}$ if $N > p$, or any $l \in [p,+\infty[$ if $p = N$,
the Sobolev space $W^{1,p}(\R^N)$ is continuously embedded in $L^l(\R^N)$, i.e.,
a constant $\sigma_l > 0$ exists such that 
\begin{equation}\label{Sob1}
|u|_l\ \le\ \sigma_l \|u\|_W \quad \hbox{for all $u \in W^{1,p}(\R^N)$}
\end{equation}
(see, e.g., \cite[Corollaries 9.10 and 9.11]{Br}). Clearly, it is $\sigma_p = 1$.
On the other hand, if $p > N$ then  $W^{1,p}(\R^N)$ is continuously imbedded in $L^\infty(\R^N)$
(see, e.g., \cite[Theorem 9.12]{Br}).

Thus, we define 
\begin{equation}\label{space}
X := W^{1,p}(\R^N) \cap L^\infty(\R^N),\qquad
\|u\|_X = \|u\|_W + |u|_\infty.
\end{equation}

From now on, we assume $1<p \le N$ as, otherwise, it is $X = W^{1,p}(\R^N)$ and
the proofs can be simplified.

\begin{lemma}\label{immergo}
For any $l \ge p$ the Banach space $X$ is continuously embedded in $L^l(\R^N)$, i.e., a constant $\sigma_l > 0$ exists such that 
\begin{equation}\label{Sob2}
|u|_l\ \le\ \sigma_l \|u\|_X \quad \hbox{for all $u \in X$.}
\end{equation}
\end{lemma}

\begin{proof}
If $p=N$ or if $p \le l \leq  p^*$ the embedding \eqref{Sob2} follows from  \eqref{Sob1} and \eqref{space}. \\
On the other hand, if $l > p^*$ then, taking any $u \in X$,
again \eqref{space} implies
\[
\int_{\R^N} |u|^l dx\ \le\ |u|_\infty^{l-p} \int_{\R^N} |u|^{p} dx \le 
|u|_\infty^{l-p} \|u\|_W^{p} \le \|u\|_X^{l},
\] 
thus \eqref{Sob2} holds with $\sigma_l =1$.
\end{proof}

From Lemma \ref{immergo} it follows that if
$(u_n)_n \subset X$, $u \in X$ are such that
$u_n \to u$ in $X$, then $u_n \to u$ also in $L^l(\R^N)$
for any $l \ge p$. This result can be weakened as follows. 

\begin{lemma}\label{immergo2}
If $(u_n)_n \subset X$, $u \in X$, $M > 0$ are such that
\begin{equation}\label{succ1}
\|u_n - u\|_W \to 0 \ \quad\hbox{if $n \to+\infty$,}
\end{equation}
\begin{equation}\label{succ2}
|u_n|_\infty \le M\quad \hbox{for all $n \in \N$,}
\end{equation}
then $u_n \to u$ also in $L^l(\R^N)$ for any $l \ge p$. 
\end{lemma}

\begin{proof}
Let $1 \le p < N$ and $l > p^*$ (otherwise, it is a direct consequence of \eqref{Sob1}). 
Then, from \eqref{space}, \eqref{succ2} and \eqref{Sob1} we have that
\[
\int_{\R^N} |u_n - u|^l dx\ \le\ |u_n - u|_\infty^{l-p} \int_{\R^N} |u_n - u|^{p} dx\ \le\ 
(M + |u|_\infty)^{l-p} \|u_n - u\|_W^{p}, 
\] 
then \eqref{succ1} implies the thesis.
\end{proof}

From now on, we consider $\, A : \R^N \times \R \times \R^N \to \R\,$
and $\, g :\R^N \times \R \to \R\,$ be such that:
\begin{itemize}
\item[$(h_0)$]
$A$  is a ${C}^1$--Carath\'eodory function, i.e., $A(\cdot, t, \xi)$ is measurable for all $(t,\xi) \in \R \times \R^N$, and $A(x,\cdot, \cdot)$ is ${C}^1$ for a.e. $x \in \R^N$ ;
\item[$(h_1)$] some positive continuous functions ${\Phi}_{i}, {{\phi}}_{i}: \R \to \R, \ i \in \{0,1,2\},$ exist such that:
\[
\begin{split}
\vert A(x,t,\xi) \vert \quad \leq & \quad {\Phi}_{0}(t){\vert t \vert}^{p} + {\phi}_{0}(t) {\vert \xi \vert}^{p} \quad \quad \quad \mbox{ a.e.  in } {\R}^{N}, \mbox{ for all } (t,\xi) \in \R \times {\R}^{N}, \\
\vert A_t(x,t,\xi) \vert \quad \leq & \quad {\Phi}_{1}(t){\vert t \vert}^{p-1} + {\phi}_{1}(t) {\vert \xi \vert}^{p} \quad \quad \mbox{ a.e.  in } {\R}^{N}, \mbox{ for all } (t,\xi) \in \R \times {\R}^{N}, \\
\vert a(x,t,\xi) \vert \quad \leq & \quad {\Phi}_{2}(t){\vert t \vert}^{p-1} + {\phi}_{2}(t) {\vert \xi \vert}^{p -1}  \quad \mbox{ a.e.  in } {\R}^{N}, \mbox{ for all } (t,\xi) \in \R \times {\R}^{N};
\end{split}
\]
\end{itemize}
\begin{itemize}
\item[$(g_0)$] $g(x,t)$ is a Carath\'eodory function;
\item[$(g_1)$] a function $ \eta \in L^{\frac{p}{p-q}}({\R}^{N})$ exists, with $1<q<p$, such that 
\[0 \leq g(x,t)t  \leq \eta(x){\vert t \vert}^{q} \qquad
\hbox{a.e. in $\R^N$, for all $t \in \R$.}\]
\end{itemize}
\begin{remark}\label{remG}
From $(g_1)$ it results 
\[\vert g(x,t) \vert \leq \eta(x){\vert t \vert}^{q-1} \qquad
\hbox{a.e. in $\R^N$, for all $t \in \R$.}\]\\
Moreover,  $(g_0)$--$(g_1)$ imply that $\displaystyle G(x,t) = \int_0^t g(x,s) ds$ is
a well defined $C^1$--Ca\-ra\-th\'eo\-do\-ry function in 
$\R^N \times \R$ and
\begin{equation}
\label{alto3}
0 \leq G(x,t) \le \frac{1}q \eta(x)|t|^{q} \qquad
\hbox{a.e. in $\R^N$, for all $t \in \R$.}
\end{equation}
\end{remark}
\begin{remark}
From $(h_1)$ it follows that
\[
A(x,0,0) = A_t(x,0,0) = 0 \quad \hbox{and}\quad a(x,0,0) = 0 \qquad 
\hbox{for a.e. $x \in \R^N$.}
\] 
Moreover, from $(g_0)$--$(g_1)$ and Remark \ref{remG} we have that
\[
G(x,0) = g(x,0) = 0 \qquad 
\hbox{for a.e. $x \in \R^N$.}
\] 
Hence, $u = 0$ is a trivial solution of \eqref{euler}.
\end{remark}

\begin{proposition}
The assumptions $(g_0)$--$(g_1)$ imply that
\begin{equation*}
 \int_{{\R}^{N}} G(x,u) dx \in \R \quad \mbox{ for all } u \in X \quad (\mbox{ or better for all } u \in W^{1, p}(\R^N))
\end{equation*}
and 
\begin{equation*}
 \int_{{\R}^{N}} g(x,u)v dx \in \R \quad \mbox{ for all } u,v \in X \quad (\mbox{ or better for all } u,v \in W^{1, p}(\R^N)).
\end{equation*}
\end{proposition}
\begin{proof}
Let $u \in W^{1,p}({\R^N}).$ As $\eta \in L^{\frac{p}{p-q}}({\R}^{N})$ and ${\vert u \vert}^q \in L^{\frac{p}{q}}({\R}^{N}),$ Hölder's inequality with $\frac{p}{p-q}$ and $\frac{p}{q}$ conjugate exponents and  \eqref{alto3} imply that
\begin{equation} \label{int1}
0\leq\int_{{\R}^{N}}G(x,u)dx \leq \frac{1}{q}\int_{{\R}^{N}} \eta(x) {\vert u \vert}^{q} dx \leq \frac{1}{q}{\vert \eta \vert_\frac{p}{p-q}}{\vert u \vert}_p^{q}.
\end{equation}
Moreover, by applying again Hölder's inequality with  $\frac{p}{p-q},  \frac{p}{q-1} \mbox{ and } p$ conjugate exponents, we have 
\begin{equation} \label{int2}
\left|  \int_{{\R}^{N}} g(x,u)v dx \right| \leq \int_{{\R}^{N}}\left| \eta(x) {\vert u \vert}^{q-1}v  \right|dx \leq {\vert \eta \vert_\frac{p}{p-q}}\vert u \vert_p^{q-1}{\vert v \vert}_p
\end{equation}
for all $u,v \in W^{1,p}({\R^N})$.  \qedhere
\end{proof}
\begin{remark} \label{rem3.8}
From $(g_0)$-$(g_1)$ we have that
\begin{equation*}
g(x,u) \in \displaystyle{L^{\frac{p}{p-1}}\big({\R}^{N}\big)} \quad \quad \mbox{ for all } u \in W^{1, p}(\R^N).
\end{equation*}
Indeed, Hölder's inequality with  $\frac{p-1}{p-q} \mbox{ and } \frac{p-1}{q-1}$ conjugate exponents implies that
\begin{equation*}
\int_{{\R}^{N}} {\vert g(x,u)\vert}^{\frac{p}{p-1}} dx \leq \displaystyle{ \vert \eta \vert_\frac{p}{p-q}^{\frac{p}{p-1}}\vert u \vert_p^{\frac{p(q-1)}{p-1}}}.
\end{equation*}
\end{remark}

Let us point out that assumptions $(h_0)$--$(h_1)$ imply that  $A(x, u, \nabla u) \in L^{1}({\R}^{N})$ for any $u \in X$.

Therefore, from \eqref{int1} it follows that the functional

\begin{equation} \label{ediff}
\J(u) = \int_{\R^N} A(x,u,\nabla u)dx  + \frac1p\ \int_{\R^N} |u|^p dx 
- \int_{\R^N} G(x,u) dx
\end{equation}
is well defined for all $u \in X$. Moreover, taking $v \in X,$ from \eqref{int2}, the G\^ateaux differential of functional $\J$ in $u$ along the direction $v$ is given by
\begin{equation}
\label{ediff1}
\begin{split}
\langle d\J(u),v\rangle &= \int_{\R^N} a(x,u,\nabla u) \cdot \nabla v dx  +\int_{\R^N} A_t(x,u,\nabla u) v dx\\& + \ \int_{\R^N} |u|^{p-2 }uvdx- \int_{\R^N} g(x,u)v dx.
 \end{split}
\end{equation}
Now, we are ready to state the following regularity result.
\begin{proposition}\label{smooth1}
Taking $p > 1$, assume that $(h_0)$--$(h_1)$, $(g_0)$--$(g_1)$ hold.
If $(u_n)_n \subset X$, $u \in X$, $M> 0$ are such that \eqref{succ1}, \eqref{succ2} hold and
\begin{equation*}\label{succ3}
 u_n \to u\quad \hbox{a.e. in $\R^N$} \ \quad\hbox{if $n \to+\infty$,}
\end{equation*}
then
\[
\J(u_n) \to \J(u)\quad \hbox{and}\quad \|d\J(u_n) - d\J(u)\|_{X'} \to 0
\quad\hbox{if $\ n\to+\infty$.}
\]
Hence, $\J$ is a $C^1$ functional on $X$ with Fr\'echet differential
defined as in \eqref{ediff1}.
\end{proposition}

\begin{proof}
It is enough to simplify the proof of Proposition 3.10 in \cite{MS_stampa} by observing that the functional $ u \in X \mapsto  \displaystyle \frac{1}{p} \int_{\R^N}|u|^{p} dx \in \R$ is of class $\mathcal{C}^{1}$.
\end{proof}


\section{Statement of the main result} \label{secmain}

From now on, we assume that in addition to hypotheses $(h_0)$--$(h_1)$ and $(g_0)$--$(g_1)$,
functions $A(x,t,\xi)$ and $g(x,t)$ satisfy the following further conditions:
\begin{itemize}
\item[$(h_2)$] a constant $\alpha_0 > 0$ exists such that
\[
A(x,t,\xi) \ge \alpha_0 |\xi|^p
\quad \hbox{a.e. in $\R^N$, for all $(t,\xi) \in \R\times\R^N$;}
\]
\item[$(h_3)$] a constant $\alpha_1 > 0$ exists such that
\[
a(x,t,\xi)\cdot \xi \ge \alpha_1 |\xi|^p
\quad \hbox{a.e. in $\R^N$, for all $(t,\xi) \in \R\times\R^N$;}
\]
\item[$(h_4)$] a constant $\eta_0$ exists such that
\[
A(x,t,\xi) \leq \eta_0 \ a(x,t,\xi)\cdot \xi
\quad \hbox{a.e. in $\R^N$, for all $(t,\xi) \in \R\times\R^N$;}
\]
\item[$(h_5)$] a constant $\alpha_2 > 0$ exists such that
\[
a(x,t,\xi)\cdot \xi + A_t(x,t,\xi) t \ge \alpha_2 a(x,t,\xi)\cdot \xi
\quad \hbox{a.e. in $\R^N$, for all $(t,\xi) \in \R\times\R^N$;}
\]
\item[$(h_6)$] some constants $\mu > p$ and $\alpha_3 > 0$ exist such that
\[
\mu A(x,t,\xi) - a(x,t,\xi)\cdot \xi - A_t(x,t,\xi) t \ge \alpha_3 A(x,t,\xi)
\quad \hbox{a.e. in $\R^N$, for all $(t,\xi) \in \R\times\R^N$;}
\]
\item[$(h_7)$]
for all $\xi$, $\xi^* \in \R^N$, $\xi \ne \xi^*$, it is
\[
[a(x,t,\xi) - a(x,t,\xi^*)]\cdot [\xi - \xi^*] > 0 
\quad\hbox{a.e. in $\R^N$, for all $t\in \R$;}
\]
\item[$(h_8)$] $\ A(x,t,\xi) = A(|x|,t,\xi)\ $ a.e. in $\R^N$, for all $t \in \R$; 
\item[$(h_9)$] some real constants $l_1, l_2, \eta_1, \eta_2$ exist such that
\[ 
\lim_{ t \to 0} \frac{{\Phi}_{1}(t)}{|t|^{\eta_1}}= l_1,
 \quad \quad \lim_{ t \to 0} \frac{{\Phi}_{2}(t)}{|t|^{\eta_2}}= l_2
\]
with ${\Phi}_{1}, {\Phi}_{2}$ as in $(h_1)$ and 
\begin{equation}\label{eta1}
\eta_1 > p\min \displaystyle \lbrace \frac{2}{N-2}, \frac{p}{N-p} \rbrace
\end{equation}
and 
\begin{equation}\label{eta2}
\eta_2 > (p-1) \min \lbrace \frac{2}{N-2}, \frac{p}{N-p}  \rbrace;
\end{equation} 
\item[$(g_2)$] $\ g(x,t) = g(|x|,t)\ $ a.e. in $\R^N$, for all $t \in \R$:
\item[$(g_3)$] the function $\eta$ introduced in $(g_1)$ is such that
\[\displaystyle \esssup_{|x| \leq 1} \eta(x) \ < \ +\infty;\]
\item[$(g_4)$] $\displaystyle{\lim_{ t \to 0^{+}} \frac{g(x,t)}{t^{p-1}}=+\infty}$ \mbox{ uniformy for a.e. $x\in \R^N, |x| \leq 1$}. 
\end{itemize}
We point out some direct consequences of the previous hypotheses.

\begin{remark}
In the assumptions $(h_2)$-$(h_3)$ we may always suppose $ \alpha_0 \leq 1$ and $ \alpha_1 \leq 1$.
\end{remark}

\begin{remark}\label{sualpha}
From $(h_3)$ and $(h_5)$ it follows that 
$\alpha_2 \le 1$. 
\end{remark}

\begin{remark}\label{conbdd}
From $(h_5)$--$(h_6)$ it follows that 
\[
(\mu - \alpha_3) A(x,t,\xi)\ \ge\ \alpha_2\ a(x,t,\xi)\cdot \xi
\quad \hbox{a.e. in $\R^N$, for all $(t,\xi) \in \R\times\R^N$;}
\]
hence, if also $(h_2)$--$(h_3)$ hold, it is 
$\alpha_3 < \mu$. So, 
\begin{equation}\label{old1}
A(x,t,\xi)\ \ge\ \alpha_4\ a(x,t,\xi)\cdot \xi
\quad \hbox{a.e. in $\R^N$, for all $(t,\xi) \in \R\times\R^N$}
\end{equation}
with $\alpha_4 = \frac{\alpha_2}{\mu - \alpha_3} >0$. Moreover,
from \eqref{old1} and $(h_6)$ we have that
\begin{equation*}\label{old2}
\mu A(x,t,\xi) - a(x,t,\xi)\cdot \xi - A_t(x,t,\xi) t\
 \ge\ \alpha_3 \alpha_4\ a(x,t,\xi)\cdot \xi
\quad \hbox{a.e. in $\R^N$, for all $(t,\xi) \in \R\times\R^N$.}
\end{equation*}
\end{remark}
\begin{remark}
We note that from $(h_3)$-$(h_6)$ it is
\[
-(1-\alpha_2)a(x,t,\xi)\cdot\xi \leq A_t(x,t, \xi)t \leq (\mu - \alpha_3)A(x,t,\xi)\leq (\mu - \alpha_3)\eta_0 a(x,t,\xi)\cdot\xi
\]
which implies that
\begin{equation} \label{contr_At_1}
\left|A_t(x,t,\xi)t\right| \leq c  a(x,t,\xi) \cdot \xi
\end{equation}
with $c=\max\lbrace (\mu - \alpha_3)\eta_0, (1-\alpha_2) \rbrace$.
\end{remark}

Now, we are able to state our main existence result.

\begin{theorem}\label{ThExist}
Assume that $(h_0)$--$(h_9)$ and $(g_0)$--$(g_4)$ hold, then problem \eqref{euler} admits at least two weak nontrivial radial bounded solutions, one negative and one positive.
\end{theorem}
We will prove Theorem \ref{ThExist} by applying Proposition \ref{Minimum Principle} to the functional $\J$ introduced in \eqref{ediff}. To this aim, the following results will be useful in the following.
\begin{proposition} \label{prop.4.4}

Assume that conditions $(h_0)$--$(h_2)$ and $(g_0)$--$(g_1)$   hold. Then, some positive constants $b_1, b_2 >0$ exist such that
\[
\J(u) \geq b_1 \|u\|_W^p - b_2 \|u\|_W^q \quad \mbox{ for any } u \in X.
\]
Hence, functional $\J$ is bounded from below, i.e., a constant $\alpha \in \R$ exists such that $$\J(u) \geq \alpha \mbox{ for any } u\in X, \mbox{ with } \alpha= \displaystyle{\min_{s\geq 0}\left(b_1 s^{p}-b_2 s^{q}\right)}. $$
\end{proposition}
\begin{proof}

From $(h_2)$ and \eqref{int1} we have
\[
\begin{split}
\J(u) &= \int_{\R^N} A(x,u,\nabla u)dx + \frac1p\ \int_{\R^N} |u|^p dx -\int_{\R^N} G(x,u) dx \\
& \geq {\alpha_0} \int_{\R^N}|\nabla u|^p dx + \frac{1}{p}\int_{\R^N} |u|^p dx - \frac{1}{q} {\vert \eta \vert}_\frac{p}{p-q}{\vert u \vert}_p^{q}\\
& \geq b_1 \|u\|_W^p-b_2 |u\|_W^q
\end{split}
\]
where $b_1=\min\{\alpha_0, \frac{1}{p}\}$ and $b_2=\frac{1}{q}{\vert \eta \vert}_\frac{p}{p-q}$.
  \qedhere
\end{proof}

\begin{lemma} \label{lemma_bmp}
Assume that  $(h_0)$, $(h_1)$, $(h_3)$ and $(h_6)$ hold. Let $(u_n)_n \subset X, u \in X$ be such that
\begin{equation} \label{conw_un}
u_n\rightharpoonup u \quad\mbox{ weakly in } W^{1, p}(\R^N),
\end{equation}
\begin{equation} \label{conqo_un}
u_n\to u \quad\mbox{  a.e. in } \R^N,
\end{equation}
\begin{equation} \label{cont_un}
\ |u_n|_\infty \leq M \quad \mbox{ for all $n \in \N$}
\end{equation}
and
\begin{equation} \label{cont_a}
\int_{\R^N} \left[a(x,u_n,\nabla u_n)-a(x,u_n, \nabla u) \right] \cdot \nabla (u_n-u) \to 0.
\end{equation}
Then,
\begin{equation} \label{cons_un}
\int_{\R^N} |\nabla u_n|^{p} dx\to \int_{\R^N} |\nabla u|^{p} dx \quad\mbox{ as $n \to +\infty$}.
\end{equation}
\end{lemma}
\begin{proof}
We will use arguments similar to those ones used in bounded domains in \cite[Lemma 4.5]{MeMu_stampa} (see also \cite[Lemma 5]{BMP}).

We will prove that any subsequence of $(u_n)_n$ admits another subsequence verifying \eqref{cons_un} and then \eqref{cons_un} holds for all sequence $(u_n)_n$.

Let $f_n$ be defined by
\[
f_n=\left[a(x,u_n,\nabla u_n)-a(x,u_n, \nabla u)\right]\cdot \nabla (u_n-u)
\]
From $(h_6)$ it follows that $f_n \geq 0 \mbox{ a.e. in $\R^N$ }$ and from \eqref{cont_a} we have $f_n \to 0 \mbox{ in $L^1(\R^N)$. } $

Thus, from \cite[Theorem 4.9]{Br} a function $\bar{h} \in L^1(\R^N)$ and a subset $Z$ of $\R^N$ exist such that $\meas(Z)=0$ and,  up to a subsequence,  
\begin{equation} \label{cond_n}
f_n(x) \to 0 \quad \mbox{ and } \quad f_n(x) \leq \bar{h}(x) < \infty \quad \quad \mbox{ for all $x \in \R^N \setminus Z$, for all $n \in \N$}.
\end{equation}

Moreover, since $u \in X$ and \eqref{conqo_un}--\eqref{cont_un} hold we can assume that
\begin{equation}\label{4.12bis}
u_n(x) \to u(x), \quad |u(x)| < +\infty \quad \mbox{ and } \quad |\nabla u(x)| < +\infty, \quad \mbox{ for all $x \in \R^N \setminus Z$ }. 
\end{equation}

From $(h_1)$ and $(h_3)$ we also have
\[
\begin{split}
f_n(x) &\geq \alpha_1 \left[\vert \nabla u_n \vert^{p} + \vert \nabla u \vert^{p}  \right] -{\Phi}_{2}(u_n){\vert u_n \vert}^{p-1} \vert \nabla u \vert - {\phi}_{2}(u_n) {\vert \nabla u_n \vert}^{p -1} \vert \nabla u \vert\\
& - {\Phi}_{2}(u){\vert u \vert}^{p-1}  \vert \nabla u_n \vert - {\phi}_{2}(u) {\vert \nabla u \vert}^{p -1} \vert \nabla u_n \vert.
\end{split}
\]
Since ${\Phi}_{2}, {\phi}_{2}$ are continuous functions, by \eqref{cont_un}, \eqref{cond_n} and \eqref{4.12bis} we find that
\[
(\nabla u_n (x))_n \hbox{ is bounded for all $x \in \R^N \setminus Z.$ }
\]
Let ${\xi}^{*}(x)$ be a cluster point of $(\nabla u_n (x))_n$. We have $\vert{\xi}^{*}(x) \vert < \infty$ and, since $f_n(x) \to 0$ and $a$ is a Carath\'eodory function , it follows that
\begin{equation*} \label{cont_a1}
\left[a(x, u, {\xi}^{*})-a(x, u, \nabla u)\right] \cdot ({\xi}^{*} - \nabla u)=0,
\end{equation*}
hence $(h_7)$ implies that $\nabla u(x)= {\xi}^{*}(x) \mbox{ for all $x \in \R^N \setminus Z$.}$\\
From this, we deduce that $\nabla u_n(x)$ converges to $\nabla u(x)$ without passing to subsequence. Hence,
\begin{equation} \label{5.20}
\nabla u_n(x) \to \nabla u(x) \quad \mbox{ for all $x \in \R^N \setminus Z.$ }
\end{equation}
Thus, from $(h_0)$, \eqref{4.12bis} and \eqref{5.20} we have that
\[
a(x,u_n(x), \nabla u_n(x)) \to a(x, u(x),\nabla u(x)) \quad \mbox{ for all $x \in \R^N \setminus Z$ }
\]
and then
\begin{equation} \label{5.23}
a(x,u_n,\nabla u_n) \cdot \nabla u_n \to a(x,u,\nabla u) \cdot \nabla u \quad \mbox{ a.e. in $\R^N$}.
\end{equation}

Now, from $(h_3)$ it is
\begin{equation} \label{5.22}
a(x,u_n,\nabla u_n) \cdot \nabla u_n \geq 0 \quad \mbox{ a.e. in $\R^N$ }.
\end{equation}

From \eqref{cont_un} and  $(h_1)$ we obtain that
\begin{equation} \label{cont_a2}
\vert a(x,u_n,\nabla u_n) \vert \leq c \left({\vert  \nabla u_n \vert}^{p -1} +{\vert u_n \vert}^{p-1}\right).
\end{equation}

Since \eqref{conw_un} holds, $u_n$ is bounded in $W^{1,p}(\R^N)$, thus from \eqref{cont_a2} the sequence $(a(x,u_n,\nabla u_n))_n$ is bounded in $L^{\frac{p}{p-1}}(\R^N)$, hence, up to subsequences, weakly converges to $a(x,u,\nabla u).$\\
It follows that
\[
\int_{\R^N} a(x,u_n,\nabla u_n) \cdot \nabla u dx \to \int_{\R^N} a(x,u,\nabla u) \cdot \nabla u dx.
\]
In a similar way, we prove that
\[
\int_{\R^N} a(x,u_n,\nabla u) \cdot \nabla u dx \to \int_{\R^N} a(x,u,\nabla u) \cdot \nabla u dx.
\]
Hence, from \eqref{cont_a} we finally find that
\begin{equation} \label{cont_a3}
\int_{\R^N} a(x,u_n,\nabla u_n) \cdot \nabla u_n dx \to \int_{\R^N} a(x,u,\nabla u) \cdot \nabla u dx.
\end{equation}
Now, we set
\[
y_n = a(x,u_n,\nabla u_n) \cdot \nabla u_n \quad \mbox{ and } \quad y = a(x,u,\nabla u) \cdot \nabla u.
\]
So, from \eqref{5.22}, \eqref{5.23}, $(h_1)$  and \eqref{cont_a3} we obtain that
\begin{equation*} 
y_n \geq 0, \quad \quad y_n \to y \quad \mbox{ a.e. in $\R^N$}, \quad \quad y \in L^1(\R^N), \quad \quad  \int_{\R^N} y_n dx \to \int_{\R^N} y dx.
\end{equation*}
From Brezis-Lieb's Lemma (see \cite{Br}) it results
\begin{equation*} \label{conv_a1}
a(x,u_n,\nabla u_n) \cdot \nabla u_n \to a(x,u,\nabla u) \cdot \nabla u \mbox{ in $L^1(\R^N)$, } 
\end{equation*}
hence, using again \cite[Theorem 4.9]{Br} a function $H \in L^1(\R^N)$ exists such that
\begin{equation} \label{cont_a4}
 a(x,u_n,\nabla u_n) \cdot \nabla u_n  \leq H(x) \mbox{ a.e. in $\R^N$. }
\end{equation}
Moreover, from $(h_3)$ and \eqref{cont_a4} we have that 
\[
\alpha_1\left(\vert \nabla u_n \vert^p \right) \leq a(x,u_n,\nabla u_n) \cdot \nabla u_n \leq H(x),
\]
thus, \eqref{cons_un} follows from \eqref{5.20} and Lebesgue's Convergence Theorem.
\end{proof} 

\begin{lemma}\label{tecnico2} 
Assume that $g(x,t)$ satisfies conditions $(g_0)$--$(g_1)$, with  $1 < q < p $, and consider $(w_n)_n$, $(v_n)_n \subset X$ and $v, w \in X$ such that
\begin{equation}\label{suglim1}
\|w_n\|_W \le M_1 \quad \hbox{for all $n \in \N$,} \qquad
w_n \to w \quad\hbox{a.e. in $\R^N$,}
\end{equation}
and
\begin{equation}\label{suglim2}
\|v_n\|_W \le M_2 \quad \hbox{for all $n \in \N$,}\qquad 
v_n \to 0 \quad\hbox{a.e. in $\R^N$,}
\end{equation}
for some constants $M_1$, $M_2 > 0$.
Then,
\begin{equation*}\label{suglim}
\lim_{n\to+\infty} \int_{\R^N} g(x,w_n) v_n dx\ =\  0.
\end{equation*}
\end{lemma}

\begin{proof}
From \eqref{suglim1}, \eqref{suglim2} and $(g_0)$ we have
\begin{equation}\label{qo}
g(x,w_n) v_n \to 0 \quad \hbox{a.e. in $\R^N$.}
\end{equation}
Moreover, from \eqref{int2} and by applying again \eqref{suglim1} and \eqref{suglim2}, it follows that
\[
\begin{split}
\int_{\R^N} \vert g(x,w_n) v_n \vert dx &\leq \vert \eta \vert_{\frac{p}{p-q}} |w_n|_p^{q-1}|v_n|_p\\
&\leq \vert \eta \vert_{\frac{p}{p-q}}\|w_n\|_W^{q-1} \|v_n\|_W \\
& \leq M_1^{q-1}M_2 \vert \eta \vert_{\frac{p}{p-q}}.
\end{split}
\]
As $\eta \in L^{\frac{p}{p-q}}({\R}^{N})$, from the absolute continuity of the Lebesgue's integral it follows that for any $\epsilon>0$ taking $\epsilon'= \left(\frac{\epsilon}{M_1^{q-1}M_2 }\right)^{\frac{p}{p-q}}$ there exists $ \delta_{\epsilon}>0$ such that
\[
\int_{A} \vert \eta \vert^{\frac{p}{p-q}} dx  \leq \epsilon'
\]
for all  measurable set $A$ with $\meas(A) < \delta_{\epsilon}$.

Thus, it follows that
\[
\int_{A} \vert g(x,w_n) v_n \vert dx \leq \epsilon
\]
for all $n \in \N$ and for all measurable set $A$  with $\meas(A) < \delta_{\epsilon}$.

On the other hand, using again the assumption $\eta \in L^{\frac{p}{p-q}}({\R}^{N})$, for any $\epsilon>0$ there exists a measurable set $B_{\epsilon}$ such that
\[
\meas(B_{\epsilon}) < \infty \ \mbox{ and } \int_{B_{\epsilon}^c} \vert \eta \vert^{\frac{p}{p-q}} dx < \epsilon'
\]
where $B_{\epsilon}^c = \R^N \setminus B_{\epsilon}$, and therefore
\[
\int_{B_{\epsilon}^c} \vert g(x,w_n)v_n \vert < \epsilon.
\]

As the sequence $\{g(x,w_n)v_n\}_n$ is uniformly integrable and verifies \eqref{qo}, the thesis follows from the Vitali's Theorem. 
\end{proof}

From now on, in order to overcame the lack of compactness of the problem we reduce to work in the space of radial functions which is a natural constraint if the problem is radially invariant (see \cite{Pa}).
Thus, in our setting, we consider the space
\[
X_r := W_r^{1,p}(\R^N) \cap L^\infty(\R^N)
\]
endowed with norm $\|\cdot\|_X$, which has dual space $(X'_r,\|\cdot\|_{X'_r})$.

The following results hold.

\begin{lemma}[Radial Lemma]\label{radiallemma}
If $N\geq 3$ and $p > 1$, every radial function $u\in W_r^{1,p}(\R^N)$ is almost everywhere
equal to a function $U(x)$, continuous for $x \ne 0$, such that
\begin{equation}
|U(x)|\ \leq\ C \frac{\|u\|_{W}}{|x|^{\vartheta}} \quad \hbox{for all $x \in \R^N$ with $|x| \geq 1$,}
\label{lemmaradiale}
\end{equation}
for suitable constants $C$, $\vartheta > 0$ depending only on $N$ and $p$. More precisely, we can take 
\begin{equation}\label{theta}
\theta= \displaystyle \max \lbrace \frac{N-2}{p}, \frac{N-p}{p} \rbrace
\end{equation}
\end{lemma}

\begin{proof}
Firstly, we recall that by classical results every function $u\in W_r^{1,p}(\R^N)$ can be assumed 
to be continuous at all points except the origin (see \cite{Br}).\\
Now, if $p \ge 2$ the proof of \eqref{lemmaradiale} with $\vartheta= \frac{N -2}{p}$, 
 follows from  \cite[Lemma A.III]{BL} 
but reasoning as in \cite[Lemma 3.5]{Pi} (see also \cite[Lemma 3.1]{CS2011}).\\
In the case $1 < p < N$, the proof follows as in \cite[Lemma 4.7]{CS_radiale}) with $\theta= \frac{N-p}{p}$. Then, \eqref{lemmaradiale} holds with $\theta$ as in \eqref{theta}.
\end{proof}

\begin{lemma}\label{immergo3}
If $p > 1$ then the following compact embeddings hold:
\[
 W_r^{1,p}(\R^N)  \hookrightarrow\hookrightarrow L^l(\R^N) \qquad\hbox{for any $p < l < p^*$.}
\]
\end{lemma}

\begin{proof}
The proof is contained essentially in \cite[Theorem 3.2]{CS2011} (see also \cite[Lemma 4.8]{CS_radiale}). 
\end{proof}

\begin{remark} \label{smooth2}
Due to the assumptions $(h_8)$ and $(g_2)$, we can reduce to look for critical points of the restriction of $\J$
in \eqref{ediff} to $X_r$, which we still denote as $\J$ for simplicity (see \cite{Pa}).\\
We recall that Proposition \ref{smooth1} 
implies that functional $\J$ is $C^1$ on the Banach space $X_r$, too, if also $(h_0)$--$(h_1)$, $(g_0)$--$(g_1)$ hold. 
\end{remark}

\section{Proof of the main result} \label{secproof}

The aim of this section is to prove that $\J$ satisfies the {\sl $(wCPS)_\beta$}-condition in $X_r$ and then to apply Proposition \ref{Minimum Principle} to the functional $\J$ on $X_r$.

In order to prove the weak Cerami--Palais--Smale condition,
we need some preliminary lemmas.

Firstly, let us point out that, while if $p > N$ the two norms $\|\cdot\|_X$ and $\|\cdot\|_W$
are equivalent, if $p \le N$ sufficient conditions are required for the boundedness of
a $W^{1,p}$--function on bounded sets as in the following result.

\begin{lemma}\label{tecnico} 
Let $\Omega$ be an open bounded domain in $\R^N$ with boundary $\partial\Omega$, 
consider $p$, $r$ so that $1 < p \le r < p^*$, $p \le N$, and take $v \in W^{1,p}(\Omega)$.
If $\gamma >0$ and $k_0\in \N$ exist such that
\[
 k_0 \ge \esssup_{\partial\Omega} v(x)
\]
and
\[
\int_{\Omega^+_k}|\nabla v|^p dx \le \gamma\left(k^r\ \meas(\Omega^+_k) +
\int_{\Omega^+_k} |v|^r dx\right)\qquad\hbox{for all $k \ge k_0$,}
\]
with $\Omega^+_k = \{x \in \Omega: v(x) > k\}$, then $\displaystyle \esssup_{\Omega} v$
is bounded from above by a positive constant which can be chosen so
that it depends only on $\meas(\Omega)$, $N$, $p$, $r$, $\gamma$, $k_0$,
$|v|_{p^*}$ ($|v|_l$ for some $l > r$ if $p^* = +\infty$). 
Vice versa, if 
\[
- k_0 \le \essinf_{\partial\Omega} v(x)
\]
and inequality
\[
\int_{\Omega^-_k}|\nabla v|^p dx \le \gamma\left(k^r\ \meas(\Omega^-_k) +
\int_{\Omega^-_k} |v|^r dx\right)\qquad\hbox{for all $k \ge k_0$,} 
\]
holds with $\Omega^-_k = \{x \in \Omega: v(x) < - k\}$, then $\displaystyle \esssup_{\Omega}(-v)$ 
is bounded from above by a positive constant which can be
chosen so that it depends only on $\meas(\Omega)$, $N$, $p$, $r$,
$\gamma$, $k_0$, $|v|_{p^*}$ ($|v|_l$ for some $l > r$ if $p^* = +\infty$).
\end{lemma}

\begin{proof}
The proof follows from \cite[Theorem  II.5.1]{LU} but reasoning as in \cite[Lemma 4.5]{CP2}.
\end{proof}

By applying Lemma \ref{tecnico}, we will prove that the weak limit in $W^{1,p}_r(\R^N)$ of
a $(CPS)_\beta$--sequence has to be bounded in $\R^N$.

For simplicity, in the following proofs, when a sequence $(u_n)_n$ is involved,
we use the notation $(\eps_n)_n$ for any infinitesimal sequence depending only on $(u_n)_n$ while $(\eps_{k,n})_n$ for any infinitesimal sequence depending not only on $(u_n)_n$ but also on some fixed integer $k$. Moreover, $c$ denotes any strictly positive constant independent of $n$ which can change from line to line.

\begin{proposition}\label{tecnico3}
Let $1 < q < p$ and assume that $(h_0)$--$(h_3)$, $(h_5)$, $(h_8)$ and $(g_0)$--$(g_3)$ hold.
Then, taking any $\beta \in \R$ and a $(CPS)_\beta$--sequence $(u_n)_n \subset X_r$, 
it follows that $(u_n)_n$ is bounded in $W_r^{1,p}(\R^N)$
and a constant $\beta_0 > 0$ exists such that 
\begin{equation}\label{c5bis}
|u_n(x)| \le \beta_0 \quad \mbox{for a.e. $x \in \R^N$ such that $|x| \ge 1$ and for all $\ n\in \N$.}
\end{equation}
Moreover, there exists $u \in X_r$ such that, up to subsequences, 
\begin{eqnarray}
&&u_n \rightharpoonup u\ \hbox{weakly in $W^{1,p}_r(\R^N)$,}
\label{c2}\\
&&u_n \to u\ \hbox{strongly in $L^l(\R^N)$ for each $l \in ]p,p^*[$,}
\label{c3}\\
&&u_n \to u\ \hbox{a.e. in $\R^N$,}
\label{c4}
\end{eqnarray}
if $n\to+\infty$.
\end{proposition}

\begin{proof}
Let $\beta \in \R$ be fixed and consider a sequence $(u_n)_n \subset X_r$
such that
\begin{equation}\label{c1}
\J(u_n) \to \beta \quad \hbox{and}\quad \|d\J(u_n)\|_{X_r'}(1 + \|u_n\|_X) \to 0\qquad
\mbox{if $\ n\to+\infty$.}
\end{equation}
From Proposition \ref{prop.4.4}, as $q<p$, $(u_n)_n$ is bounded in $W_r^{1,p}(\R^N)$ and therefore Lemma \ref{radiallemma} implies the uniform estimate \eqref{c5bis}. Furthermore, $u \in W_r^{1,p}(\R^N)$ exists such that \eqref{c2}--\eqref{c4} hold, up to subsequences. \\
Now, we have just to prove that $u \in L^\infty(\R^N)$.\\
Clearly, \eqref{c5bis} and \eqref{c4} imply
\begin{equation}
\label{ess0}
 \esssup_{|x| \ge 1} |u(x)| \ \leq \beta_0  \ < \ +\infty.
\end{equation}
Then, it is enough to prove that 
\begin{equation}
\label{ess00}
\esssup_{|x| \le 1} |u(x)| \ < \ +\infty.
\end{equation}
Arguing by contradiction, let us assume that either
\begin{equation}
\label{ess1}
\esssup_{|x| \le 1} u(x) \ = \ +\infty
\end{equation}
or
\begin{equation}
\label{ess2}
\esssup_{|x| \le 1} (- u(x)) \ = \ +\infty.
\end{equation}
If, for example, \eqref{ess1} holds 
then, for any fixed $k \in\N$, $k > \beta_0$ we have that
\begin{equation}\label{asp}
\meas(B^+_k) > 0 \quad\hbox{with}\quad B^+_k = \{x \in B_1(0): u(x) > k\}.
\end{equation}
We note that the choice of $k$ and \eqref{ess0} imply that
\begin{equation}\label{asp1}
B^+_k \ =\ \{x \in \R^N: u(x) > k\}.
\end{equation}
Moreover, if we set 
\[
 B^+_{k,n}\ =\ \{x \in B_1(0): u_n(x) > k\},\quad n \in \N,
\]
the choice of $k$ and \eqref{c5bis} imply that
\begin{equation}\label{asp2}
B^+_{k,n} \ =\ \{x \in \R^N: u_n(x) > k\} \qquad \hbox{for all $n \in \N$.}
\end{equation}
Now, consider the new function
$R^+_k : t \in\R \to R^+_kt \in \R$ such that
\[
R^+_kt = \left\{\begin{array}{ll}
0&\hbox{if $t \le k$}\\
t-k&\hbox{if $t > k$}
\end{array}\right. .
\]
By definition and \eqref{asp1}, respectively \eqref{asp2}, it results
\begin{equation}\label{asp3}
R^+_k u(x)\ =\ \left\{\begin{array}{ll}
0&\hbox{if $x \not\in B_k^+$}\\
u(x) - k&\hbox{if $x \in B_k^+$}
\end{array}\right. ,\quad
R^+_k u_n(x)\ =\ \left\{\begin{array}{ll}
0&\hbox{if $x \not\in B_{k,n}^+$}\\
u_n(x) - k&\hbox{if $x \in B_{k,n}^+$}
\end{array}\right. .
\end{equation}
Clearly, \eqref{c5bis}, \eqref{ess0} and $k > \beta_0$ imply
\begin{equation}\label{asp4}
R^+_k u \in W_0^{1,p}(B_1(0))\qquad \hbox{and}\qquad R^+_k u_n \in W_0^{1,p}(B_1(0)) \quad \hbox{for all $n \in N$.}
\end{equation}
From \eqref{c2} it follows that $R_k^+u_n \rightharpoonup R_k^+u$
weakly in $W^{1,p}_r(\R^N)$, then, from \eqref{asp4}, in $W_0^{1,p}(B_1(0))$.
As $W_0^{1,p}(B_1(0)) \hookrightarrow\hookrightarrow L^l(B_1(0))$ for any $1 \le l< p^*$, then 
\begin{equation}\label{b01}
\lim_{n\to+\infty} \int_{B_1(0)} |R_k^+u_n|^l dx\ =\  \int_{B_1(0)} |R_k^+u|^l dx \quad \hbox{for any $1 \le l< p^*$.}
\end{equation}
Moreover, from \eqref{c3} we have $u_n \to u$ strongly in $L^l(B_1(0))$ for any $l \in ]p,p^*[$
and then 
\begin{equation}\label{b011}
\lim_{n\to+\infty} \int_{B_1(0)} |u_n|^l dx\ =\  \int_{B_1(0)} |u|^l dx \quad \hbox{for any $1 \le l< p^*$.}
\end{equation}
Thus, by the sequentially weakly lower semicontinuity of the norm $\|\cdot\|_{W}$,
 we have that
\[
\int_{\R^N} |\nabla R_k^+u|^p dx + \int_{\R^N} |R_k^+u|^p dx\ \le\ 
\liminf_{n\to+\infty}\left(\int_{\R^N} |\nabla R_k^+u_n|^p dx + \int_{\R^N} |R_k^+u_n|^p dx\right),
\]
i.e., from \eqref{asp3} -- \eqref{b01} we have
\[
\begin{split}
\int_{B^+_k} |\nabla u|^p dx + \int_{B_1(0)} |R_k^+u|^p dx\ &\le\ 
\liminf_{n\to+\infty}\left(\int_{B^+_{k,n}} |\nabla u_n|^p dx + \int_{B_1(0)} |R_k^+u_n|^p dx\right)\\
&=\ \liminf_{n\to+\infty} \int_{B^+_{k,n}} |\nabla u_n|^p dx + \int_{B_1(0)} |R_k^+u|^p dx.
\end{split}
\]
Hence,
\begin{equation}\label{b1}
\int_{B^+_k} |\nabla u|^p dx \ \le\ \liminf_{n\to+\infty} \int_{B^+_{k,n}} |\nabla u_n|^p dx .
\end{equation}
On the other hand, since $\|R^+_ku_n\|_X \le \|u_n\|_X$ holds,  it follows that
\[
|\langle d\J(u_n),R^+_ku_n\rangle| \le \|d\J(u_n)\|_{X'_r}\|u_n\|_X,
\]
then \eqref{c1} and \eqref{asp} imply that $n_{k}\in \N$ exists such that
\begin{equation}\label{b2}
\langle d\J(u_n),R^+_ku_n\rangle < \meas(B^+_k) \qquad \hbox{for all $n \ge n_{k}$.}
\end{equation}
Let us point out that, being $\alpha_2 \leq 1$, assumption $(h_5)$ implies that
\[
\begin{split}
\langle d\J(u_n),R^+_ku_n\rangle\ &=\ 
\int_{B^+_{k,n}} a(x,u_n,\nabla u_n) \cdot \nabla u_n dx + \int_{B^+_{k,n}} A_t(x,u_n,\nabla u_n) (u_n-k) dx  \\
&+ \int_{B^+_{k,n}} |u_n|^{p-2}u_n (u_n-k) dx - \int_{B^+_{k,n}} g(x,u_n)R^+_ku_n dx \\ &= \int_{B^+_{k,n}}\left(1-\frac{k}{u_n}\right) \left[a(x,u_n,\nabla u_n) \cdot \nabla u_n +  A_t(x,u_n,\nabla u_n)u_n \right] dx\\
&+\int_{B^+_{k,n}}\frac{k}{u_n} a(x,u_n,\nabla u_n) \cdot \nabla u_n dx + \int_{B^+_{k,n}} |u_n|^{p-2}u_n (u_n-k) dx   - \int_{B^+_{k,n}} g(x,u_n)R^+_ku_n dx\\
& \geq \alpha_2 \int_{B^+_{k,n}} a(x,u_n,\nabla u_n)\cdot \nabla u_n dx - \int_{B^+_{k,n}} g(x,u_n)R^+_ku_n dx.
\end{split}
\]

Hence, from the previous inequalities and $(h_3)$ it follows that
\begin{equation}\label{b3}
\begin{split}
\alpha_1 \alpha_2 \int_{B^+_{k,n}} |\nabla u_n|^p dx\ & \le\ \langle d\J(u_n),R^+_ku_n\rangle + \int_{B^+_{k,n}} g(x,u_n)R^+_ku_n dx.
\end{split}
\end{equation}
Now, from \eqref{asp4}, \eqref{b01} and $(g_1)$ we get that
\begin{equation} \label{controllo_g}
\lim_{n\to+\infty} \int_{\R^N} g(x,u_n)R^+_ku_n dx = \int_{\R^N} g(x,u)R^+_ku dx.
\end{equation}

Thus, from \eqref{b1}--\eqref{controllo_g} and $(g_3)$ we obtain that
\begin{equation*}\label{b4}
\begin{split}
\int_{B^+_{k}} |\nabla u|^p dx\ &\le\
c\left(\meas(B^+_k) + \int_{B^+_k} g(x,u) R^+_ku\ dx\right)\\ 
&\leq c \ \meas(B^+_k) + c \int_{B^+_k} \eta(x) \vert u \vert^q dx \leq \bar{c}\left(\meas(B^+_k)+ \int_{B^+_{k}}\vert u \vert^p\right)
\end{split}
\end{equation*}
with $\displaystyle \bar{c}=\max\{c,\displaystyle \esssup_{|x| \leq 1} \eta(x)\}$ since 
\[
\int_{B^+_{k}} \eta(x) \vert u \vert^q dx \leq \int_{B^+_{k}} \eta(x) \vert u \vert^p dx \leq \displaystyle \esssup_{|x| \leq 1} \eta(x) \int_{B^+_{k}} \vert u \vert^p dx
\] 
as $q<p$ and $u(x)>1$ for all $x \in B^+_{k}$.
Thus, we get
\[
\int_{B^+_{k}} |\nabla u|^p dx \leq \displaystyle \bar{c} \left(\meas(B^+_k)+\int_{B^+_{k}}\vert u \vert^p\right).
\]
As this inequality holds for all $k > \beta_0$, 
Lemma \ref{tecnico} implies that \eqref{ess1} is not true. 
Thus, \eqref{ess2} must
hold. In this case, fixing any $k \in\N$, $k > \beta_0$, we have
\[
\meas(B^-_k) > 0,\qquad \hbox{with $B^-_k= \{x \in B_1(0): u(x) < - k\}$,}
\]
and we can consider
$R^-_k : t \in\R \to R^-_kt \in \R$ such that
\[
R^-_kt = \left\{\begin{array}{ll}
0&\hbox{if $t \ge -k$}\\
t+k&\hbox{if $t <- k$}
\end{array}\right. .
\]
Thus, reasoning as above, but replacing $R^+_k$ with $R^-_k$, and again by means
of Lemma \ref{tecnico} we prove  that \eqref{ess2} cannot hold.
Hence, \eqref{ess00} has to be true.
\end{proof}
Now, we are ready to prove the $(wCPS)$ condition in $\R$ by adapting the arguments developed 
in \cite[Proposition 3.4]{CP1}, or also \cite[Proposition 4.6]{CP2}, 
to our setting in the whole space $\R^N$.

\begin{proposition}\label{wCPS}
If $1 < q < p$ and $(h_0)$--$(h_9)$, $(g_0)$--$(g_3)$ hold, then 
functional $\J$ satisfies the weak Cerami--Palais--Smale condition in $X_r$
at each level $\beta \in \R$.
\end{proposition}

\begin{proof}
Let $\beta \in \R$ be fixed and consider a sequence $(u_n)_n \subset X_r$ verifying \eqref{c1}.
By applying Proposition \ref{tecnico3} the uniform estimate \eqref{c5bis} holds 
and there exists $u \in X_r$ such that, 
up to subsequences, \eqref{c2}--\eqref{c4} are satisfied.

We need to prove the following three steps:
\begin{itemize}
\item[1.] defining $T_k : \R \to \R$ such that
\begin{equation}\label{troncok}
T_kt = \left\{\begin{array}{ll}
t&\hbox{if $|t| \le k$}\\
k\frac t{|t|}&\hbox{if $|t| > k$}
\end{array}\right. ,
\end{equation}
with $k \ge |u|_{\infty} + 1$ then, as $n \to +\infty$, we have
\begin{equation}\label{pp}
\J(T_ku_n) \to \beta 
\end{equation}
and 
\begin{equation}\label{p2}
\|d\J(T_ku_n)\|_{X_r'} \to 0;
\end{equation}
\item[2.] $\|u_n - u\|_{W} \to 0$ if $n\to+\infty$, as
\begin{equation}\label{eq4}
\|T_ku_n - u\|_{W} \to 0 \qquad \hbox{as $n \to +\infty$;}
\end{equation}
\item[3.] $\J(u) = \beta$ and $d\J(u) = 0$.
\end{itemize}
\smallskip

\noindent
{\sl Step 1.} 
Taking any $k > \max\{|u|_\infty, \beta_0\}$,
if we set 
\begin{equation}\label{asp511}
 B_{k,n}\ =\ \{x \in B_1(0): |u_n(x)| > k\},\quad n \in \N,
\end{equation}
the choice of $k$ and \eqref{c5bis} imply that
\begin{equation}\label{asp5}
B_{k,n} \ =\ \{x \in \R^N: |u_n(x)| > k\} \qquad \hbox{for all $n \in \N$.}
\end{equation}
Then, from \eqref{troncok} and \eqref{asp5} we have that
\begin{equation}\label{asp51}
T_k u_n(x)\ =\ \left\{\begin{array}{ll}
u_n(x)&\hbox{for a.e. $x \not\in B_{k,n}$}\\
k\frac{u_n(x)}{|u_n(x)|}&\hbox{if $x \in B_{k,n}$}
\end{array}\right. 
\end{equation}
and
\[
|T_ku_n|_\infty \le k, \quad \|T_ku_n\|_{W} \le \|u_n\|_{W}\qquad \hbox{for each $n \in \N$.}
\]
Defining $R_k : \R \to \R$ such that
\[
R_kt = t - T_kt = \left\{\begin{array}{ll}
0&\hbox{if $|t| \le k$}\\
t-k\frac t{|t|}&\hbox{if $|t| > k$}
\end{array}\right.,
\]
from \eqref{asp5} it results
\begin{equation}\label{asp51bis}
R_k u_n(x)\ =\ \left\{\begin{array}{ll}
0&\hbox{for a.e. $x \not\in B_{k,n}$}\\
u_n(x) - k\frac{u_n(x)}{|u_n(x)|}&\hbox{if $x \in B_{k,n}$}
\end{array}\right. ;
\end{equation}
hence, \eqref{asp511} and \eqref{asp51bis} imply that
\begin{equation}\label{asp7}
R_k u_n \in W_0^{1,p}(B_1(0)) \qquad \hbox{for all $n \in N$.}
\end{equation}
Since $k > |u|_\infty$, we deduce that 
\[
T_ku(x) = u(x) \quad \hbox{and}\quad R_ku(x) = 0 \qquad \hbox{for a.e. $x \in \R^N$;} 
\]
thus, from \eqref{c2} it follows that $R_ku_n \rightharpoonup 0$
weakly in $W^{1,p}_r(\R^N)$, and, from \eqref{asp7}, in $W_0^{1,p}(B_1(0))$.
From the compact embedding of $W_0^{1,p}(B_1(0))$ in $L^l(B_1(0))$ for any $1 \le l< p^*$, we have that
\begin{equation}\label{b012}
\lim_{n\to+\infty} \int_{\R^N} |R_ku_n|^l dx\ =\  0 \quad \hbox{for any $1 \le l< p^*$.}
\end{equation}
Now, arguing as in the proof of \eqref{b3} but replacing $R^+_ku_n$ with $R_ku_n$ we obtain 
\begin{equation}\label{b8}
\begin{split}
{\alpha_1 \alpha_2} \int_{B_{k,n}} |\nabla u_n|^p dx\ &\le\
{\alpha_2} \int_{B_{k,n}} a(x,u_n, \nabla u_n) \cdot \nabla u_n dx \\
& \le\ \langle d\J(u_n),R_ku_n\rangle + \int_{B_{k,n}} g(x,u_n) R_ku_n dx.
\end{split}
\end{equation}
We note that  \eqref{c1} and $\|R_ku_n\|_X\leq \|u_n\|_X$ imply that
\begin{equation}\label{lim12}
\lim_{n\to+\infty}|\langle d\J(u_n),R_ku_n\rangle|\ =\ 0;
\end{equation}
while the boundedness of the sequences $(\|u_n\|_W)_n$ and $(\|R_ku_n\|_W)_n$, \eqref{c4}, \eqref{ess0}, \eqref{asp51bis} and Lemma \ref{tecnico2} imply that
\begin{equation}\label{lim13}
\lim_{n\to+\infty}\int_{B_{k,n}} g(x,u_n) R_ku_n dx \ =\ 0.
\end{equation}
From \eqref{b8}--\eqref{lim13} we obtain that
\begin{equation}\label{b111}
 \lim_{n\to+\infty} \int_{B_{k,n}} |\nabla u_n|^p dx\ =\ 0
\end{equation}
and 
\begin{equation}\label{b113}
\lim_{n\to+\infty} \int_{B_{k,n}} a(x,u_n,\nabla u_n)\cdot \nabla u_n dx \ =\ 0.
\end{equation}
Hence, from \eqref{asp51bis}, \eqref{b012} and \eqref{b111} it follows that
\begin{equation}\label{b112}
 \lim_{n\to+\infty} \|R_k u_n\|_W\ =\ 0.
\end{equation}
Moreover, from \eqref{c4}, \eqref{asp511} and $k > \vert u \vert_{\infty}$ we obtain
\begin{equation}
\label{lim6}
\lim_{n\to+\infty} \meas(B_{k,n})\ =\ 0,
\end{equation}
which together \eqref{b011} implies
\begin{equation}
\label{lim7}
\lim_{n\to+\infty} \int_{B_{k,n}} |u_n|^l dx\ =\ 0\qquad \hbox{for any $1 \le l< p^*$.}
\end{equation}
From \eqref{ediff} and \eqref{asp51} we have
\begin{equation}\label{vai}
\begin{split}
\J(T_ku_n) = & \int_{\R^N \setminus B_{k,n}} A(x,u_n,\nabla u_n) dx+\int_{B_{k,n}} A\left(x,k \frac{u_n}{\vert u_n \vert},0\right) dx \\ &+\frac1p\int_{\R^N \setminus B_{k,n}} |u_n|^p dx +\frac{1}{p}\int_{ B_{k,n}} k^p dx- \int_{\R^N} G(x,T_ku_n) dx\\
=& \J(u_n) -\int_{B_{k,n}} A(x,u_n,\nabla u_n) dx +\int_{B_{k,n}} A\left(x,k \frac{u_n}{\vert u_n \vert},0\right) dx \\ &-\frac{1}{p}\int_{ B_{k,n}} |u_n|^p dx+\frac{1}{p}\int_{ B_{k,n}} k^p dx - \int_{\R^N} (G(x,T_ku_n) - G(x,u_n)) dx
\end{split}
\end{equation}
where, from $(h_4)$ and \eqref{b113} we have
\begin{equation} \label{1}
\int_{B_{k,n}} A(x,u_n,\nabla u_n) dx \leq \eta_0 \int_{B_{k,n}} a(x,u_n,\nabla u_n)\cdot\nabla u_n dx\to 0,
\end{equation}
while $(h_1)$ and \eqref{c4} imply
\begin{equation} \label{2}
\begin{split}
\int_{B_{k,n}} A\left(x,k \frac{u_n}{\vert u_n \vert},0\right) dx \leq \int_{B_{k,n}} {\Phi}_{0}\left(k \frac{u_n}{\vert u_n \vert}\right){k}^{p} dx &\leq \displaystyle \Big({\max_{|t| \leq k} {\Phi}_{0}(t)}\Big)k^{p} \meas{B_{k,n}} \to 0
\end{split}
\end{equation}
and \eqref{lim6}--\eqref{lim7} imply
\begin{equation}\label{lim_nuovo}
-\frac{1}{p}\int_{ B_{k,n}} |u_n|^p dx +\frac{1}{p} \int_{ B_{k,n}} k^p dx \to 0.
\end{equation}
Furthermore, from $\eqref{asp51}$ it is
\begin{equation}
\label{lim8}
\displaystyle\int_{\R^N} (G(x,T_ku_n) - G(x,u_n)) dx=\int_{B_{k,n}} \left(G\left(x,k\frac{u_n}{\vert u_n \vert}\right) - G(x,u_n)\right) dx \to 0
\end{equation}
since \eqref{int1}, \eqref{lim6} and \eqref{lim7} imply that
\[
\int_{B_{k,n}} G\left(x,k\frac{u_n}{\vert u_n \vert}\right) dx\leq\frac{1}{q} \vert \eta \vert_{\frac{p}{p-q}} k^{q}\meas(B_{k,n})\to 0
\]
and
\[
\int_{B_{k,n}} G(x,u_n) dx \leq\frac{1}{q}{\vert \eta \vert_{\frac{p}{p-q}}}\int_{B_{k,n}}{\vert u_n \vert}^q dx\to 0.
\]
Then, \eqref{pp} follows from \eqref{c1} and \eqref{vai}–\eqref{lim8}. \\
In order to prove \eqref{p2}, we take $v \in X_r$ such that $\|v\|_X = 1$; hence,
$|v|_\infty \le 1$, $\|v\|_{W} \le 1$.

From \eqref{ediff1} and \eqref{asp51} we have that
\[
\begin{split}
\langle d\J(T_ku_n),v\rangle\ & =\  \int_{\R^N} a(x,T_ku_n,\nabla T_ku_n) \cdot \nabla v dx + \int_{\R^N} A_t(x,T_ku_n,\nabla T_ku_n) v dx\\
&+ \int_{\R^N}|T_ku_n|^{p-2}T_ku_n v dx - \int_{\R^N} g(x,T_ku_n) v dx\\
& =\int_{\R^N \setminus B_{k,n}} a(x,u_n,\nabla u_n) \cdot \nabla v dx
 +\int_{B_{k,n}} a \left(x, k \frac{u_n}{\vert u_n \vert},0\right)\cdot \nabla v\\
 &+ \int_{\R^N \setminus B_{k,n}} A_t(x,u_n,\nabla u_n) v dx +\int_{B_{k,n}} A_t \left(x, k \frac{u_n}{\vert u_n \vert},0\right)v dx\\
&+\int_{\R^N \setminus B_{k,n}}|u_n|^{p-2}u_n v dx + \int_{B_{k,n}}k^{p-1} \frac{u_n}{|u_n|}v dx  - \int_{\R^N} g(x,T_ku_n) v dx\\
&=\  \langle d\J(u_n),v\rangle - \int_{B_{k,n}} a(x,u_n,\nabla u_n)\cdot \nabla v dx
 - \int_{B_{k,n}} A_t(x,u_n,\nabla u_n) v dx\\
&- \int_{B_{k,n}}|u_n|^{p-2}u_n v dx + \int_{B_{k,n}} (g(x,u_n) - g(x,T_ku_n)) v dx+\epsilon_n,
\end{split}
\]
since $(h_1)$, \eqref{lim6}, H\"older inequality and $|\nabla v|_{p}\leq 1, |v|_{\infty}\leq 1$  imply  that
\begin{equation} \label{contr_a}
\begin{split}
\displaystyle \left| \int_{B_{k,n}} a \left(x, k \frac{u_n}{\vert u_n \vert},0\right) \cdot \nabla v dx \right|	 &\leq  \int_{B_{k,n}}{\Phi}_{2}\left(k \frac{u_n}{\vert u_n \vert}\right){k}^{p-1}|\nabla v|dx\\
& \leq \displaystyle \Big({\max_{|t| \leq k} {\Phi}_{2}(t)}\Big)\left(\int_{B_{k,n}} {k}^{p} dx\right)^{\frac{p-1}{p}} \to 0,
\end{split}
\end{equation}
\begin{equation} \label{contr_At}
\begin{split}
\displaystyle \left| \int_{B_{k,n}} A_t \left(x, k \frac{u_n}{\vert u_n \vert},0\right) v dx \right|	 &\leq \int_{B_{k,n}} {\Phi}_{1}\left(k \frac{u_n}{\vert u_n \vert}\right){k}^{p-1}dx \\
&\leq \displaystyle \Big({\max_{|t| \leq k} {\Phi}_{1}(t)}\Big) {k}^{p-1} \meas(B_{k,n}) \to 0,
\end{split}
\end{equation}
\begin{equation} \label{un2}
\left|\int_{B_{k,n}} k^{p-1} \frac{u_n}{|u_n|} v dx\right| \le k^{p-1} \meas(B_{k,n}) \ \to\ 0,
\end{equation}

where all the limits hold uniformly with respect to $v$.

Furthermore, from \eqref{contr_At_1} and \eqref{b113} we have that
\begin{equation*} \label{lim2}
\lim_{n\to+\infty} \int_{B_{k,n}} \left| A_t(x,u_n,\nabla u_n) u_n \right|dx\ =\ 0,
\end{equation*}
and then, since $1\leq k \leq \vert u_n \vert$ and $\vert v \vert_{\infty} \leq 1$, we get
\begin{equation} \label{contr_At_2}
\begin{split}
\left| \int_{B_{k,n}} A_t(x,u_n,\nabla u_n) v dx\right| &\leq \int_{B_{k,n}}\left| A_t(x,u_n,\nabla u_n)\right|dx\\ &\leq \int_{B_{k,n}}\left| A_t(x,u_n,\nabla u_n)\right| \vert u_n \vert dx \to 0
\end{split}
\end{equation}
 uniformly with respect to $v$, while from \eqref{lim7}, H\"older inequality and $\vert v\vert_p \leq 1$ we have
 \begin{equation*} \label{un1}
\left|\int_{B_{k,n}} |u_n|^{p-2}u_n v dx\right|\ \le\
 \left(\int_{B_{k,n}} |u_n|^{p} dx\right)^{\frac{p-1}{p}}\ \to\ 0.
\end{equation*}
Moreover, from \eqref{int2}, \eqref{lim6} - \eqref{lim7} and $|v|_p\leq 1$ it results
\[
\left| \int_{B_{k,n}} g(x,u_n) v dx \right| \leq {\vert \eta \vert_{\frac{p}{p-q}}} \left(\int_{B_{k,n}} \vert u_n \vert^p dx\right)^{\frac{q-1}{p}} \to 0 \quad \mbox{ uniformly with respect to $v$ }
\]
and
\[
\left| \int_{B_{k,n}} g(x,T_ku_n) v dx \right| \leq {\vert \eta \vert_{\frac{p}{p-q}}} \left(\int_{B_{k,n}} \vert T_ku_n \vert^p dx\right)^{\frac{q-1}{p}} \to 0 \quad \mbox{ uniformly with respect to $v.$ }
\]
Thus, summing up, from \eqref{c1} we obtain 
\begin{equation}\label{eq1}
|\langle d\J(T_ku_n),v\rangle| \le \eps_{k,n} + \left|\int_{B_{k,n}} a(x,u_n, \nabla u_n) \cdot \nabla v dx\right|.
\end{equation}
Now, in order to estimate the last integral in \eqref{eq1}, following 
the notations introduced in the proof of Proposition \ref{tecnico3}, let us
consider the set $B^+_{k,n}$ and the test function defined as
\[
\varphi^+_{k,n} = v R^+_{k}u_n.
\]
By definition, we have
\[
\|\varphi^+_{k,n}\|_X \le 2 \|u_n\|_X,
\]
and, thus, \eqref{c1} implies 
\[
\|d\J(u_n)\|_{X_r'}\|\varphi^+_{k,n}\|_X\ \le\ \eps_n.
\]
From definition \eqref{asp3} and direct computations we note that
\[
\begin{split}
\langle d\J(u_n),\varphi^+_{k,n}\rangle\ &=\  \int_{B^+_{k,n}} a(x,u_n,\nabla u_n) R^+_{k}u_n \cdot \nabla v dx +  \int_{B^+_{k,n}} a(x,u_n,\nabla u_n) \cdot v \nabla u_n dx\\
&+ \int_{B^+_{k,n}}A_t(x,u_n,\nabla u_n)v R^+_{k}u_n dx + \int_{B^+_{k,n}}|u_n|^{p-2}u_n vR^+_{k}u_n dx  - \int_{B^+_{k,n}} g(x,u_n)v R^+_{k}u_n dx
\end{split}
\]
where, since $B_{k,n}^+ \subset B_{k,n}$, from \eqref{lim6} we have
\[
\lim_{n\to+\infty} \meas(B_{k,n}^+)\ =\ 0,
\]
while $\vert v \vert_{\infty} \leq 1$, \eqref{b113}, \eqref{contr_At_2}, \eqref{lim7} and \eqref{int2} imply
\[
\begin{split}
\left|\int_{B^+_{k,n}} a(x,u_n,\nabla u_n) \cdot v \nabla u_n dx\right|\ &\le\  \int_{B_{k,n}^+} a(x,u_n,\nabla u_n) \cdot \nabla u_n dx \ \to\ 0, \\
\left|\int_{B^+_{k,n}}A_t(x,u_n,\nabla u_n)v R^+_{k}u_ndx\right|\ &\le\  
 \int_{B_{k,n}^+} |A_t(x,u_n,\nabla u_n)|(u_n-k)dx \\ &\leq \int_{B_{k,n}^+} |A_t(x,u_n,\nabla u_n)|u_n dx \ \to\ 0,\\
\left|\int_{B^+_{k,n}}|u_n|^{p-2}u_n vR^+_{k}u_n dx\right| &\leq \int_{B^+_{k,n}}|u_n|^{p} dx \to 0, \\
\left|\int_{B^+_{k,n}} g(x,u_n)v R^+_{k}u_n dx\right| & \le \int_{B_{k,n}^+}|g(x,u_n)| |u_n| dx \leq  {\vert \eta \vert_{\frac{p}{p-q}}}\left(\int_{B_{k,n}^+}|u_n|^{p}\right)^{\frac{q-1}{p}} \to\ 0
 \end{split}
\]
uniformly with respect to $v$. From the previous estimates it follows that 
\begin{equation}\label{limitea}
\lim_{n \to +\infty} \int_{B^+_{k,n}} a(x,u_n,\nabla u_n) R^+_{k}u_n \cdot \nabla v dx=0
\end{equation}

Now, if we fix $k > \max\{|u|_\infty, \beta_0\} + 1$, all the previous computations hold also for $k-1$ and then in particular, \eqref{b111}, \eqref{lim7} and \eqref{limitea} become
\begin{equation}\label{b111bis}
\lim_{n\to+\infty} \int_{B_{k-1,n}} |\nabla u_n|^p dx\ =\ 0, \quad \lim_{n\to+\infty} \int_{B_{k-1,n}} |u_n|^p dx\ =\ 0
\end{equation}
and
\begin{equation}\label{bbis1}
\lim_{n \to +\infty} \int_{B^+_{k-1,n}} a(x,u_n,\nabla u_n) R^+_{k-1}u_n \cdot \nabla v dx=0.
\end{equation}
From \eqref{bbis1} since $B^+_{k,n} \subset B^+_{k-1,n}$, it is 
\[
\begin{split}
\epsilon_{k,n}=\int_{B^+_{k-1,n}} a(x,u_n,\nabla u_n) R^+_{k-1}u_n \cdot \nabla v dx &= \int_{B^+_{k,n}} a(x,u_n,\nabla u_n) R^+_{k-1}u_n \cdot \nabla v dx \\
&+\int_{B^+_{k-1,n} \setminus B^+_{k,n}} a(x,u_n,\nabla u_n) R^+_{k-1}u_n \cdot \nabla v dx\\
&= \int_{B^+_{k,n}} a(x,u_n,\nabla u_n) R^+_{k}u_n \cdot \nabla v dx + \int_{B^+_{k,n}} a(x,u_n,\nabla u_n) \cdot \nabla v dx\\
&+\int_{B^+_{k-1,n} \setminus B^+_{k,n}} a(x,u_n,\nabla u_n) R^+_{k-1}u_n \cdot \nabla v dx
\end{split}
\]
where $(h_1)$, \eqref{asp3}, the properties of $B^+_{k-1,n} \setminus B^+_{k,n}$, H\"older inequality, $\vert \nabla v \vert_p \leq 1$  and \eqref{b111bis} imply
\[
\begin{split}
&\left|\int_{B^+_{k-1,n} \setminus B^+_{k,n}} a(x,u_n,\nabla u_n) R^+_{k-1}u_n \cdot \nabla v\ dx\right| \leq k \int_{B^+_{k-1,n} \setminus B^+_{k,n}} |a(x,u_n,\nabla u_n)||\nabla v|\ dx \\
& \leq k \max_{\vert t \vert \leq k} \Phi_2(t) \int_{B^+_{k-1,n} \setminus B^+_{k,n}}|u_n|^{p-1} |\nabla v| \ dx + k\max_{\vert t \vert \leq k} \phi_2(t)  \int_{B^+_{k-1,n} \setminus B^+_{k,n}}|\nabla u_n|^{p-1} |\nabla v|\ dx \\
& \leq k\max_{\vert t \vert \leq k} \Phi_2(t)  \left(\int_{B^+_{k-1,n} \setminus B^+_{k,n}} |u_n|^{p} dx\right)^{\frac{p-1}{p}}  + k\max_{\vert t \vert \leq k} \phi_2(t)  \left(\int_{B^+_{k-1,n} \setminus B^+_{k,n}} |\nabla u_n|^{p} dx\right)^{\frac{p-1}{p}} \to 0.
\end{split}
\]
The previous arguments imply
\begin{equation}\label{bbis}
\left|\int_{B^+_{k,n}} a(x,u_n,\nabla u_n)  \cdot \nabla v dx\right|\ \le\  \eps_{k,n}.
\end{equation}
Similar arguments apply also if we consider $B^-_{k,n}$
and the test functions
\[
\varphi^-_{k,n} = v R^-_{k}u_n,\quad \varphi^-_{k-1,n} = v R^-_{k-1}u_n;
\]
hence, we have
\begin{equation}\label{eq3bis}
\left|\int_{B^{-}_{k,n}}   a(x,u_n,\nabla u_n) \cdot \nabla v  dx\right|\le \eps_{k,n}.
\end{equation}
Thus, \eqref{p2} follows from \eqref{eq1}, \eqref{bbis} and \eqref{eq3bis} as all $\eps_{k,n}$ are independent of $v$.

\noindent
{\sl Step 2.}
We note that \eqref{c2}--\eqref{c4} imply that, if $n\to+\infty$,
\begin{eqnarray*}
&&T_k u_n \rightharpoonup u\ \hbox{weakly in $W^{1,p}_r(\R^N)$,}
\label{cc2}\\
&&T_ku_n \to u\ \hbox{strongly in $L^l(\R^N)$ for each $l \in ]p,p^*[$,}
\label{cc3}\\
&&T_ku_n \to u\ \hbox{a.e. in $\R^N$.}
\label{cc4}
\end{eqnarray*}
Now, arguing as in \cite{AB1}, let us consider the real map 
\[
\psi: t \in \R \mapsto \psi(t) = t \e^{\bar{\eta} t^2}\in \R,
\] 
where $\displaystyle \bar{\eta} > \left(\frac\beta{2\alpha}\right)^2$ will be fixed once $\alpha$, $\beta > 0$ are chosen in a suitable way later on. By definition,
\begin{equation}\label{psi1}
\alpha \psi'(t) - \beta |\psi(t)| > \frac\alpha 2\qquad \hbox{for all $t \in \R$.}
\end{equation}
If we define $v_{k,n} = T_ku_n - u$, since $k> |u|_{\infty}$, we have that $|v_{k,n}|_\infty \le 2k$ for all $n \in \N$.
Therefore,
\begin{equation}\label{psi2}
|\psi(v_{k,n})| \le \psi(2k),\quad 0<\psi'(v_{k,n}) \le \psi'(2k)
\qquad\hbox{a.e. in $\R^N$ for all $n\in \N$,}
\end{equation}
and
\begin{equation}\label{psi3}
\psi(v_{k,n}) \to 0, \quad
\psi'(v_{k,n}) \to 1 \qquad\hbox{a.e. in $\R^N$ if $n\to +\infty$.}
\end{equation}
Furthermore, we note that
\[
|\psi(v_{k,n})|	\ \le\ |v_{k,n}| \e^{4 k^2 \bar{\eta}}
\qquad\hbox{a.e. in $\R^N$ for all $n\in \N$,}
\]
thus, direct computations imply that $(\|\psi(v_{k,n})\|_X)_n$ is bounded, and so from \eqref{psi3}, up to subsequences, it is 
\begin{equation}\label{psi5}
\psi(v_{k,n}) \rightharpoonup 0 \quad \hbox{weakly in $W^{1,p}_r(\R^N)$,}
\end{equation}
while from \eqref{p2} it follows that
\[
\langle d\J(T_ku_n),\psi(v_{k,n})\rangle \to 0 \quad \hbox{as $n\to +\infty$,}
\]
where
\[
\begin{split}
\langle d\J(T_ku_n),\psi(v_{k,n})\rangle\ & =\  \int_{\R^N\setminus B_{k,n}} a(x,u_n,\nabla u_n) \cdot \nabla \psi(v_{k,n}) dx +\int_{B_{k,n}} a \left(x, k \frac{u_n}{\vert u_n \vert},0\right) \cdot\nabla \psi(v_{k,n}) dx \\ 
 &+ \int_{\R^N \setminus B_{k,n}} A_t(x,u_n,\nabla u_n)\psi(v_{k,n})  dx +\int_{B_{k,n}} A_t \left(x, k \frac{u_n}{\vert u_n \vert},0\right)\psi(v_{k,n}) dx\\
&+\int_{\R^N \setminus B_{k,n}} |u_n|^{p-2}u_n \psi(v_{k,n}) \ dx+ \int_{B_{k,n}}k^{p-1}\frac{u_n}{|u_n|}\psi(v_{k,n}) \ dx\\
&- \int_{\R^N} g(x,T_ku_n) \psi(v_{k,n}) dx.
\end{split}
\]
Since $(\|\psi(v_{k,n})\|_X)_n$ is bounded, arguing as in \eqref{contr_a}--\eqref{un2} it follows that  
\[
\lim_{n\to+\infty}\int_{B_{k,n}}a \left(x, k \frac{u_n}{\vert u_n \vert},0\right)\cdot \nabla \psi(v_{k,n}) dx=0,
\]
\[
\lim_{n\to+\infty} \int_{B_{k,n}} A_t \left(x, k \frac{u_n}{\vert u_n \vert},0\right)\psi(v_{k,n}) dx \ =\ 0,
\]
\[
\lim_{n\to+\infty} \int_{B_{k,n}}k^{p-1}\frac{u_n}{|u_n|}\psi(v_{k,n}) \ dx=0.
\]

Furthermore, from Lemma \ref{tecnico2} with $w_n=T_ku_n$ and $v_n = \psi(v_{k,n})$,
it follows that
\[
\lim_{n\to+\infty} \int_{\R^N} g(x,T_ku_n) \psi(v_{k,n}) dx\ =\ 0.
\]
Hence, summing up, the previous relations imply that
\begin{equation} \label{a}
\begin{split}
\eps_{k,n} & = \  \int_{\R^N \setminus B_{k,n}} a(x,u_n,\nabla u_n) \psi'(v_{k,n}) \cdot \nabla v_{k,n} dx\\
& + \int_{\R^N \setminus B_{k,n}} A_t(x,u_n,\nabla u_n) \psi(v_{k,n}) dx+\int_{\R^N \setminus B_{k,n}} |u_n|^{p-2}u_n \psi(v_{k,n}) \ dx.
\end{split}
\end{equation}
We note that from $(h_1)$ it is
\begin{equation} \label{b}
\begin{split}
& \left|\int_{\R^N \setminus B_{k,n}} A_t(x,u_n,\nabla u_n) \psi(v_{k,n}) dx\right|\  \\&\le\ \int_{\R^N \setminus B_{k,n}} \left({\Phi_{1}}(u_n)|u_n|^{p-1}+\max_{\vert t \vert \leq k}{\phi}_{1}(t)|\nabla u_n|^p \right)  |\psi(v_{k,n})|dx.
\end{split}
\end{equation}
We prove that
\begin{equation} \label{c}
\lim_{n\to+\infty} \int_{\R^N \setminus B_{k,n}} {\Phi_{1}}(u_n)|u_n|^{p-1} |\psi(v_{k,n})| dx =0.
\end{equation}
In fact, since the sequence $(u_n)_n$ is bounded in $W^{1,p}_r(\R^N)$, there exists a constant $M>0$ such that 
\[
\vert u_n\vert_W \leq M, \quad \|u_n-u\|_W \leq M \quad \mbox{ for any $n \in \N$.}
\]
Moreover, from assumption $(h_9)$ it is $$\lim_{ t \to 0} \frac{{\Phi}_{1}(t)}{|t|^{\eta_1}}= l_1 \mbox{ with } l_1 \geq 0, $$
hence, there exists $\delta_1>0$ such that
\begin{equation} \label{561_n}
{\Phi}_{1}(t) < (l_1+1)|t|^{\eta_1} \quad \quad \mbox{ for any $t \in \R, |t|<\delta_1$}.
\end{equation}
 Now, fixing $\epsilon>0$, as from \eqref{eta1} and \eqref{theta} it is $(\eta_1+p)\theta>N$, there exists $R_{\epsilon}>1$ such that
\begin{equation} \label{(i)}
\frac{CM}{R_{\epsilon}^{\theta}}<\delta_1
\end{equation} 
and 
\begin{equation}\label{(ii)}
(l_1+1)(CM)^{p+\eta_1}\e^{\bar{\eta}\frac{C^2M^2}{{R_{\epsilon}}^2{\theta}}}\int_{{B^c_{R_\epsilon}}}\frac{1}{|x|^{(\eta_1+p)\theta}}dx< \epsilon
\end{equation}
where $C$ is the constant introduced in \eqref{lemmaradiale}.
 From \eqref{lemmaradiale} and \eqref{(i)}, it follows that
 \[
|u_n(x)|\leq C \frac{M}{|x|^{\theta}} \leq C \frac{M}{R_{\epsilon}^{\theta}}< \delta_1 \quad \quad \mbox{ for all $x \in \R^N$ with $|x|>R_{\epsilon}$};
 \]
 hence, \eqref{561_n}, \eqref{lemmaradiale} and \eqref{(ii)} imply
\[
\begin{split}
\int_{(\R^N \setminus B_{k,n})\cap {{B^c_{R_\epsilon}}}} {\Phi_{1}}(u_n) |u_n|^{p-1} |\psi(v_{k,n})| dx &\leq\int_{(\R^N \setminus B_{k,n})\cap {{B^c_{R_\epsilon}}}} (l_1+1)|u_n|^{\eta_1+p-1} |u_n-u| \e^{\bar{\eta}\frac{\|u_n-u\|^2_W}{|x|^{2\theta}}} dx\\
&\leq (l_1+1)(CM)^{\eta_1+p}\e^{\bar{\eta}\frac{C^2M^2}{{R_{\epsilon}}^2{\theta}}}\int_{{B^c_{R_\epsilon}}}\frac{1}{|x|^{(\eta_1+p)\theta}}dx\\
&< \epsilon
\end{split}
\]
while from H\"older's inequality
\[
\begin{split}
&\int_{(\R^N \setminus B_{k,n}) \cap B_{R_\epsilon}} {\Phi_{1}}(u_n)  |u_n|^{p-1} |\psi(v_{k,n})| dx\\ 
&\leq \max_{|t|\leq k}{\Phi_1(t)}\left(\int_{(\R^N \setminus B_{k,n}) \cap B_{R_\epsilon}} |u_n|^{p} dx\right)^{\frac{1}{p}}\left(\int_{(\R^N \setminus B_{k,n}) \cap B_{R_\epsilon}}|\psi(v_{k,n})|^p dx\right)^{\frac{1}{p}} \to 0
\end{split}
\]
since \eqref{psi5} implies that $\psi(v_{k,n}) \to 0$ in $L{^p_{loc}}(\R^N)$. Then, \eqref{c} follows.

Moreover, from $(h_3)$ it follows that
\begin{equation} \label{d}
\begin{split}
\int_{\R^N \setminus B_{k,n}} |\nabla u_n|^{p} |\psi(v_{k,n})| dx &\leq \frac{1}{\alpha_1}\int_{\R^N \setminus B_{k,n}} a(x,u_n,\nabla u_n) \cdot \nabla u_n|\psi(v_{k,n})| dx \\
&= \frac{1}{\alpha_1} \int_{\R^N \setminus B_{k,n}}a(x,u_n,\nabla u_n) \cdot \nabla v_{k,n}|\psi(v_{k,n})| dx\\&+ \frac{1}{\alpha_1} \int_{\R^N \setminus B_{k,n}}a(x,u_n,\nabla u_n) \cdot \nabla u |\psi(v_{k,n})|dx.
\end{split}
\end{equation}

where the boundedness of $(u_n)_n$ in $W^{1,p}_r(\R^N)$, $(h_1)$, H\"older's inequality, 
\eqref{psi3} and the Lebesgue's Dominated Convergence Theorem imply that
\begin{equation} \label{e}
\begin{split}
&\left| \int_{\R^N \setminus B_{k,n}} a(x,u_n,\nabla u_n) \cdot \nabla u |\psi(v_{k,n})|dx\right| \\
&\leq  \int_{\R^N \setminus B_{k,n}} {\Phi}_{2}(u_n){\vert u_n\vert}^{p-1}|\nabla u ||\psi(v_{k,n})|dx  + \int_{\R^N \setminus B_{k,n}} {\phi}_{2}(u_n) {\vert \nabla u_n \vert}^{p -1} |\nabla u||\psi(v_{k,n})|dx \\ 
&\leq \left( \max_{\vert t \vert \leq k} \Phi_2(t)\right)|u_n|_{p}^{p-1} \left(\int_{\R^N \setminus B_{k,n}} |\nabla u |^{p}|\psi(v_{k,n})|^{p}dx\right)^{\frac{1}{p}}\\& +  \left(\max_{\vert t \vert \leq k} \phi_2(t)\right) |\nabla u_n|^{p-1}_p \Big(\int_{\R^N \setminus B_{k,n}}|\nabla u|^p |\psi(v_{k,n}|^p) dx\Big)^{\frac{1}{p}} \to 0.
\end{split}
\end{equation}
From \eqref{a}--\eqref{e} we obtain
\[
\begin{split}
\epsilon_{k,n} & \geq \int_{\R^N \setminus B_{k,n}} a(x, u_n, \nabla u_n)\psi'(v_{k,n}) \cdot \nabla v_{k,n}dx \\
&-\frac{1}{\alpha_1}\max_{|t|\leq k}{\phi_1(t)} \int_{\R^N \setminus B_{k,n}}a(x, u_n, \nabla u_n) \cdot \nabla v_{k,n} |\psi(v_{k,n})| dx + \int_{\R^N \setminus B_{k,n}}|u_n|^{p-2}u_n \psi(v_{k,n}) dx.\\
\end{split}
\]
Thus, setting
\[
h_{k,n}(x) = \psi'(v_{k,n})
-\ \max_{|t|\leq k}{\phi_1(t)} |\psi(v_{k,n})|,
\]
and choosing, in the definition of function $\psi$, constants $\alpha = 1$ and $\beta = \max_{|t|\leq k}{\phi_1(t)}$, from \eqref{psi1} it results
\begin{equation}\label{psi6}
h_{k,n}(x) \ > \ \frac12 \quad \hbox{a.e. in $\R^N$.}
\end{equation}
Therefore, we have that
\begin{equation}\label{estimate1}
\begin{split}
\eps_{k,n}\  &\ge \ \int_{\R^N \setminus B_{k,n}} h_{k,n} a(x,u_n, \nabla u_n) \cdot \nabla v_{k,n} dx
+ \int_{\R^N \setminus B_{k,n}} |u_n|^{p-2}u_n \psi(v_{k,n}) dx\\
&=\  \int_{\R^N \setminus B_{k,n}} a(x,u,\nabla u) \cdot \nabla v_{k,n} dx\\
& + \int_{\R^N \setminus B_{k,n}} h_{k,n}\left( a(x,u_n,\nabla u_n) - a(x,u_n,\nabla u)\right)\cdot \nabla v_{k,n} dx\\
&+ \int_{\R^N \setminus B_{k,n}} \left(h_{k,n} a(x,u_n,\nabla u) - a(x,u,\nabla u)\right) \cdot \nabla v_{k,n} dx\\
& + \int_{\R^N \setminus B_{k,n}} (|u_n|^{p-2}u_n - |u|^{p-2}u) \psi(v_{k,n}) dx + \int_{\R^N \setminus B_{k,n}} |u|^{p-2}u \psi(v_{k,n}) dx,
\end{split}
\end{equation}

where \eqref{c2}, respectively \eqref{psi5} imply that
\[
\lim_{n\to+\infty}\int_{\R^N \setminus B_{k,n}} a(x,u, \nabla u) \cdot \nabla v_{k,n} dx = 0,\qquad
\lim_{n\to+\infty} \int_{\R^N \setminus B_{k,n}} |u|^{p-2}u \psi(v_{k,n}) dx = 0.
\]
On the other hand, we note that $(h_0)$, \eqref{c4} and \eqref{psi3} infer that
$$h_{k,n} a(x,u_n,\nabla u) - a(x,u,\nabla u) \to 0 \quad \quad\mbox{ a.e. in $\R^N$}$$ 
while from H\"older's inequality it follows that
\begin{equation}\label{1}
\begin{split}
& \left|\int_{\R^N \setminus B_{k,n}} \Big(h_{k,n} a(x,u_n,\nabla u) - a(x,u,\nabla u)\Big)\cdot \nabla v_{k,n} dx\right| \\
& \quad \leq \Big(\int_{\R^N \setminus B_{k,n}} {\left|h_{k,n} a(x,u_n,\nabla u) - a(x,u,\nabla u)\right|}^{\frac{p}{p-1}} dx\Big)^{\frac{p-1}{p}}|\nabla v_{k,n}|_p.
\end{split}
\end{equation}
From \eqref{psi2} and $(h_1)$ we have that for any $x \in (\R^N \setminus B_{k,n})$
\begin{equation}\label{569nuova}
\begin{split}
&\vert h_{k,n} a(x,u_n,\nabla u) - a(x,u,\nabla u) \vert^{\frac{p}{p-1}} \\
 &\leq \left((\psi'(2k))\Phi_2(u_n)|u_n|^{p-1}+(\max_{|t|\leq k}{\phi_2(u_n)} )|\nabla u|^{p-1} \vert + |a(x,u, \nabla u)|\right)^{\frac{p}{p-1}}\\
 &\leq c\left((\Phi_2(u_n))^{\frac{p}{p-1}}|u_n|^{p} +|\nabla u|^p + |u|^{p}\right).
\end{split}
\end{equation}
Now, from $(h_9)$ it is $$\lim_{ t \to 0} \frac{{\Phi}_{2}(t)}{|t|^{\eta_2}}= l_2 \mbox{ with } l_2 \geq 0, $$
hence, there exists $\delta_2>0$ such that
\begin{equation*} 
{\Phi}_{2}(t) < (l_2+1)|t|^{\eta_2} \quad \quad \mbox{ for any $t \in \R, |t|<\delta_2$}.
\end{equation*}
Therefore, taking $\bar{R}>0$ such that $\frac{CM}{{\bar{R}}^{\theta}}< \delta_2$ and using \eqref{lemmaradiale} it is
\begin{equation} \label{570nuova}
\begin{split}
(\Phi_2(u_n))^{\frac{p}{p-1}}|u_n|^{p} &\leq (l_2+1)^{\frac{p}{p-1}}|u_n|^{\frac{\eta_2 p}{p-1} +p}\\
&\leq (l_2+1)^{\frac{p}{p-1}} \left(\frac{CM}{|x|^{\theta}}\right)^{(\eta_2 + p-1)\frac{p}{p-1}}\quad  \mbox{ for all $x \in \R^N$ with $|x|> \bar{R}$}.
\end{split}
\end{equation}
Now from \eqref{eta2} and \eqref{theta} it is $({\eta_2+p-1})\frac{p\theta}{p-1}>N$ , hence, recalling that $u \in W^{1,p}(\R^N)$, for any $\epsilon >0$ there exists $R_{\epsilon}>\bar{R}$ such that the previous estimates \eqref{569nuova} and \eqref{570nuova} imply
\begin{equation}\label{2}
\int_{(\R^N \setminus B_{k,n})\cap B^c_{R_\epsilon}} {\left|h_{k,n} a(x,u_n,\nabla u) - a(x,u,\nabla u)\right|}^{\frac{p}{p-1}} dx < \epsilon.
\end{equation}
Moreover
\[
{\left|h_{k,n} a(x,u_n,\nabla u) - a(x,u,\nabla u)\right|}^{\frac{p}{p-1}} \leq b_5 +b_6|\nabla u|^p \quad \mbox{ a.e. in $\R^N$, }
\]
hence the Lebesgue's Dominated Convergence Theorem gives us that
\begin{equation}\label{3}
\lim_{n\to+\infty}\int_{(\R^N \setminus B_{k,n})\cap B_{R_\epsilon}(0)}|h_{k,n} a(x,u_n,\nabla u) - a(x,u,\nabla u)|^{\frac{p}{p-1}} dx = 0.
\end{equation}

Then, as $(\|v_{k,n}\|_{W})_n$ is bounded, \eqref{1},\eqref{2} and \eqref{3} imply
\[
\lim_{n\to+\infty}\int_{\R^N \setminus B_{k,n}} \Big(h_{k,n} a(x,u_n,\nabla u) - a(x,u,\nabla u)\Big)\cdot \nabla v_{k,n} dx=0
\]

Thus, from \eqref{psi6}--\eqref{estimate1}, by using the previous estimate,  the strong convexity of the power function with exponent $p > 1$, $(h_7)$ and $\displaystyle{\e^{\bar{\eta} v_{k,n}^2}} \ge 1$ we obtain
\[
\begin{split}
\eps_{k,n}\  &\ge \ \frac{1}{2} \int_{\R^N \setminus B_{k,n}} \Big(a(x,u_n,\nabla u_n)-a(x,u_n,\nabla u)\Big) \cdot \nabla (u_n-u) dx \\ &+ \int_{\R^N \setminus B_{k,n}} (|u_n|^{p-2}u_n - |u|^{p-2}u) (u_n-u) dx.
\end{split}
\]
Using again $(h_7)$ and the strong convexity of the power function with exponent $p > 1$ we have
\[
\lim_{n\to +\infty}\int_{\R^N \setminus B_{k,n}}   \Big(a(x,u_n,\nabla u_n)-a(x,u_n,\nabla u)\Big) \cdot \nabla (u_n-u) dx =0
\]
and
\[
\lim_{n\to +\infty}\int_{\R^N \setminus B_{k,n}}  (|u_n|^{p-2}u_n - |u|^{p-2}u) (u_n-u)  dx=0
\]

Hence, $u_n \to u$ in $L^p(\R^N)$ and by applying Lemma \ref{lemma_bmp}, we get that $\nabla u_n \to \nabla u$ in $L^p(\R^N)$; thus, $\|u_n-u\|_W \to 0$.

{\sl Step 3.} 
As from definition it is $|T_ku_n|_{\infty}\leq k$ for all $n\in \N$, the proof follows from \eqref{eq4}, \eqref{cc4}, Proposition \ref{smooth1}, \eqref{pp} and \eqref{p2}.
\end{proof}

\begin{proof}[Proof of the Theorem $\ref{ThExist}$] 
The functional $\J$ is bounded from below in $X$ (see Proposition \ref{prop.4.4}) and satisfies condition $(wCPS)$ in $\R$ (see Proposition \ref{wCPS}), thus, from Proposition \ref{Minimum Principle}, $\J$ admits a minimum point $u^*$ in $X$. Clearly, it is
\[
\J(u^*)=\min_{u\in X}\J(u) \leq \J(0)=0.
\] 
In order to prove that $u^*$ is not trivial since $\J(u^*)<0$, we consider $\varphi_1 \in W_{0}^{1,p}(B_1(0))$ the unique eigenfunction associated to the first eigenvalue $\lambda_1$ of $-\Delta_{p}$ in $B_1(0)$ \hbox{ (see \cite{Lin}) } such that 
\[
\begin{split}
&\varphi_1>0 \hbox{ a.e. in $B_1(0)$, } \
\varphi_{1} \in L^\infty(B_1(0)), \\ 
 & \int_{B_1(0)}|\varphi_{1}|^p dx =1, \ \int_{B_1(0)}|\nabla \varphi_{1}|^p dx =\lambda_1
\end{split}
\]
and we also denote by $\varphi_{1}$ its null extension in $\R^N$.
Taking $\tau \in \left(0,1\right)$, from $(h_1)$ we have
\[
\begin{split}
\J(\tau\varphi_1)&= \int_{\R^N} A(x,\tau\varphi_1, \nabla(\tau \varphi_1))dx +\frac{1}{p}\int_{\R^N}|\tau \varphi_1|^p dx-\int_{\R^N} G(x,\tau\varphi_1)dx \\
&\leq \int_{B_1(0)}\left({\Phi}_{0}(\tau \varphi_1(x)){\vert \tau \varphi_1(x) \vert}^{p} + {\phi}_{0}(\tau \varphi_1(x)) {\vert \nabla (\tau \varphi_1(x)) \vert}^{p}\right) dx \\ & +\frac{\tau^p}{p}\int_{B_1(0)}|\varphi_1|^p dx -\int_{\Omega} G(x,\tau\varphi_1)dx \\& \leq c_1 \tau^p - \int_{B_1(0)} G(x,\tau\varphi_1)dx,
\end{split}
\]
where
$c_1= \displaystyle{\max_{0 \leq t \leq |\varphi_{1}|_{\infty}} {\Phi}_{0}(t)} + \lambda_1 \displaystyle{\max_{0 \leq t \leq |\varphi_{1}|_{\infty}} {\phi}_{0}(t)} +\frac{1}{p}$.

Now, from assumption $(g_4)$ there exists a constant $\delta>0$ such that for any $s \in [0,\delta]$ and for a.e. $x \in B_1(0)$ it is $G(x,s)> 2 c_1 s^p.$ \\ Then, for any $\tau>0$ small enough, in particular $ 0<\tau<\frac{\delta}{|\varphi_{1}|_{\infty}}$, it results
\[
\J(\tau\varphi_1) \leq c_1 \tau^p - 2c_1 \tau^p < 0.  
\]
Finally, let us prove that $\J$ has at least two solutions, one negative and one positive.  For this, let us denote by $u_{+}=\max \{0, u\}$ and  $u_{-}= \max \{0, -u\}$, the positive and the negative part of $u$, respectively, so that $u=u_{+}-u_{-}$. 

If we replace $g(x,u)$ by $g_{+}(x,u):=g(x,u_{+})$, all the previous statements still hold true for the functional $\J_{+}$ obtained by replacing $G$ with $G_{+}$, defined as $\displaystyle{G_{+}(x,t)=\int_0^t g_{+}(x,s) ds}$. In particular, $\J_{+}$ has a nontrivial critical point $u$. Hence, from $(h_5)$ and $(h_3)$ we find that
\[
\begin{split}
0&= \left\langle d\J_{+}(u), -u_{-}\right\rangle= \int_{\R^N}a(x, -u_-,\nabla (-u_-)) \nabla (-u_-) dx +\int_{\R^N} A_t(x, -u_-,\nabla (-u_-))(-u_-)dx \\
&+\int_{\R^N}|u_-|^{p}dx+ \int_{\R^N} g_{+}(x,u)u_-dx\\
&\geq \alpha_1\alpha_2 \int_{\R^N} \vert \nabla u_{-}\vert^p dx+\int_{\R^N} |u_-|^{p} dx\\
&\geq \alpha_1\alpha_2 \|u_{-}\|_{W}.
\end{split}
\]
Hence, $u_{-}=0 \mbox{ a.e. in $\R^N$}$ and $u$ is a positive critical point of $\J$.

Similarly, replacing $g(x,u)$ with $g(x,-u_-)$, we find a negative solution of the equation \eqref{euler}.
\qedhere
\end{proof}


\end{document}